\documentclass[11pt,a4paper]{article}
\usepackage{amsmath}
\usepackage{amsthm}
\usepackage{amssymb}
\usepackage{amscd}
\usepackage{graphicx}
\usepackage{epsfig}
\usepackage[matrix,arrow,curve]{xy}
\usepackage{color}
\usepackage{mathrsfs}
\usepackage[scr=rsfs,cal=boondox]{mathalpha}

\usepackage{bbm}
\usepackage{stmaryrd}
\usepackage{hyperref}

\usepackage{latexsym}
\usepackage{amsfonts}
\input xy

\newtheorem{theorem}{Theorem}
\newtheorem{proposition}[theorem]{Proposition}

\newtheorem{corollary}[theorem]{Corollary}

\theoremstyle{definition}
\newtheorem{example}[theorem]{Example}
\newtheorem{remark}[theorem]{Remark}
\newtheorem{definition}[theorem]{Definition}

\font\black=cmbx10 \font\sblack=cmbx7 \font\ssblack=cmbx5 \font\blackital=cmmib10  \skewchar\blackital='177
\font\sblackital=cmmib7 \skewchar\sblackital='177 \font\ssblackital=cmmib5 \skewchar\ssblackital='177
\font\sanss=cmss11 \font\ssanss=cmss9 scaled 900 \font\sssanss=cmss8 scaled 600 \font\blackboard=msbm10
\font\sblackboard=msbm7 \font\ssblackboard=msbm5 \font\caligr=eusm10 \font\scaligr=eusm7 \font\sscaligr=eusm5

\font\bsymb=cmsy10 scaled\magstep2
\def\all#1{\setbox0=\hbox{\lower1.5pt\hbox{\bsymb
       \char"38}}\setbox1=\hbox{$_{#1}$} \box0\lower2pt\box1\;}
\def\exi#1{\setbox0=\hbox{\lower1.5pt\hbox{\bsymb \char"39}}
       \setbox1=\hbox{$_{#1}$} \box0\lower2pt\box1\;}

\def\tx#1{{\fam0\relax#1}}

\newfam\bifam
\textfont\bifam=\blackital \scriptfont\bifam=\sblackital \scriptscriptfont\bifam=\ssblackital

\newfam\blfam
\textfont\blfam=\black \scriptfont\blfam=\sblack \scriptscriptfont\blfam=\ssblack

\newfam\bbfam
\textfont\bbfam=\blackboard \scriptfont\bbfam=\sblackboard \scriptscriptfont\bbfam=\ssblackboard

\newfam\ssfam
\textfont\ssfam=\sanss \scriptfont\ssfam=\ssanss \scriptscriptfont\ssfam=\sssanss
\def\sss#1{{\fam\ssfam\relax#1}}

\newfam\clfam
\textfont\clfam=\caligr \scriptfont\clfam=\scaligr \scriptscriptfont\clfam=\sscaligr

\def\pmb#1{\setbox0\hbox{${#1}$} \copy0 \kern-\wd0 \kern.2pt \box0}
\def\pmbb#1{\setbox0\hbox{${#1}$} \copy0 \kern-\wd0
      \kern.2pt \copy0 \kern-\wd0 \kern.2pt \box0}
\def\pmbbb#1{\setbox0\hbox{${#1}$} \copy0 \kern-\wd0
      \kern.2pt \copy0 \kern-\wd0 \kern.2pt
    \copy0 \kern-\wd0 \kern.2pt \box0}
\def\pmxb#1{\setbox0\hbox{${#1}$} \copy0 \kern-\wd0
      \kern.2pt \copy0 \kern-\wd0 \kern.2pt
      \copy0 \kern-\wd0 \kern.2pt \copy0 \kern-\wd0 \kern.2pt \box0}
\def\pmxbb#1{\setbox0\hbox{${#1}$} \copy0 \kern-\wd0 \kern.2pt
      \copy0 \kern-\wd0 \kern.2pt
      \copy0 \kern-\wd0 \kern.2pt \copy0 \kern-\wd0 \kern.2pt
      \copy0 \kern-\wd0 \kern.2pt \box0}

\mathchardef\za="710B  
\mathchardef\zb="710C  
\mathchardef\zg="710D  
\mathchardef\zd="710E  
\mathchardef\zve="710F 
\mathchardef\zz="7110  
\mathchardef\zh="7111  
\mathchardef\zvy="7112 
\mathchardef\zi="7113  
\mathchardef\zk="7114  
\mathchardef\zl="7115  
\mathchardef\zm="7116  
\mathchardef\zn="7117  
\mathchardef\zx="7118  
\mathchardef\zp="7119  
\mathchardef\zr="711A  
\mathchardef\zs="711B  
\mathchardef\zt="711C  
\mathchardef\zu="711D  
\mathchardef\zvf="711E 
\mathchardef\zq="711F  
\mathchardef\zc="7120  
\mathchardef\zw="7121  
\mathchardef\ze="7122  
\mathchardef\zy="7123  
\mathchardef\zf="7124  
\mathchardef\zvr="7125 
\mathchardef\zvs="7126 
\mathchardef\zf="7127  
\mathchardef\zG="7000  
\mathchardef\zD="7001  
\mathchardef\zY="7002  
\mathchardef\zL="7003  
\mathchardef\zX="7004  
\mathchardef\zP="7005  
\mathchardef\zS="7006  
\mathchardef\zU="7007  
\mathchardef\zF="7008  
\mathchardef\zW="700A  

\newcommand{\be}{\begin{equation}}
\newcommand{\ee}{\end{equation}}
\newcommand{\raa}{\rightarrow}

\newcommand{\bea}{\begin{eqnarray}}
\newcommand{\eea}{\end{eqnarray}}
\newcommand{\beas}{\begin{eqnarray*}}
\newcommand{\eeas}{\end{eqnarray*}}
\def\*{{\textstyle *}}

\newcommand{\ot}{\otimes}

\newcommand{\pa}{\partial}
\newcommand{\ti}{\times}

\def\ran{\rangle}

\def\cA{{\cal A}}
\def\cD{{\cal D}}

\def\cL{{\cal L}}

\def\cO{{\cal O}}
\def\cU{{\cal U}}

\def\cV{\mathcal{V}}

\def\Sec{\sss{Sec}}

\def\sA{{\sss A}}

\def\sJ{{\sss J}}

\def\sT{{\sss T}}
\def\sV{{\sss V}}

\def\sv{{\sss v}}

\def\sj{{\sss j}}

\def\xd{\tx{d}}

\def\dt{\xd_{\sss T}}
\def\dts{\xd_{\sss T^*}}

\newdir{|>}{%
!/4.5pt/@{|}*:(1,-.2)@^{>}*:(1,+.2)@_{>}}
\newdir{ (}{{}*!/-5pt/@^{(}}

\newcommand{\la}{\langle}

\newcommand{\Z}{\mathbb{Z}}
\newcommand{\R}{\mathbb{R}}

\newcommand{\n}{\nabla}

\newcommand{\op}[1]{\!\!\mathop{\rm ~#1}\nolimits}

\newcommand{\id}{\op{id}}

\def\Rt{{\R^\ti}}
\def\Lt{{L^\ti}}

\def\mh{\mathcal{h}}

\newcommand{\we}{\wedge}

\begin{document}
\title{\bf Contact geometric mechanics:\\ the Tulczyjew triple}
\date{}
\author{\\ Katarzyna  Grabowska$^1$\\ Janusz Grabowski$^2$ 
        \\ \\
         $^1$ {\it Faculty of Physics}\\
                {\it University of Warsaw}\\
                \\$^2$ {\it Institute of Mathematics}\\
                {\it Polish Academy of Sciences}
                }
\maketitle

\begin{abstract}\noindent We propose a generalization of the classical Tulczyjew triple as a geometric tool in Hamiltonian and Lagrangian formalisms which serves for contact manifolds. The r\^ole of the canonical symplectic structures on cotangent bundles in the Tulczyjew's case is played by the canonical contact structures on the bundles $\sJ^1L$ of first jets of sections of line bundles $L\to M$. Contact Hamiltonians and contact Lagrangians are understood as sections of certain line bundles, and they determine (generally implicit) dynamics on the contact phase space $\sJ^1L$. We study also a contact analog of the Legendre map and the Legendre transformation of generating objects in both contact formalisms. Several explicit examples are offered.

\par

\end{abstract}

\tableofcontents

\maketitle


\section{Introduction}
The classical \emph{Tulczyjew triple} is the following commutative diagram of canonical vector bundle morphisms:
\begin{equation}\label{e:0}\xymatrix{
\sT^\ast\sT^\ast M\ar[rd]_{\pi_{\sT^\ast M}} & & \sT\sT^\ast M\ar[ll]_{\beta_M}\ar[rr]^{\alpha_M}\ar[ld]^{\tau_{\sT^\ast M}}\ar[rd]_{\sT\pi_M} & & \sT^\ast\sT M\ar[ld]^{\pi_{\sT M}}  \\
 & \sT^\ast M\ar[rd]_{\pi_M} & & \sT M\ar[ld]^{\tau_M} & \\
 & & M & & .
}
\end{equation}
The top line consists of isomorphisms of \emph{double vector bundles} (for double vector bundles we refer to \cite{Grabowski:2009,Konieczna:1999,Pradines:1974}). These double vector bundles are canonically symplectic manifolds and $\za_M$ is simultaneously a symplectomorphism, while $\zb_M$ is an anti-symplectomorphism (the latter depends on convention). Starting with local coordinates $(x^i)$ in $M$, we obtain the adapted local coordinates in all parts of the triple, in which the maps $\zb_M$ and $\za_M$ read
$$\zb_M(x^i,p_j,\dot x^k,\dot p_l)=(x^i,p_j,-\dot p_k,\dot x^l)\,,\quad
\za_M(x^i,p_j,\dot x^k,\dot p_l)=(x^i,\dot x^j,\dot p_k,p_l).
$$
The composition
\begin{align}\label{ci}R_{\sT M}=\zb_M\circ \za_M^{-1}:\sT^*\sT M\to\sT^*\sT^*M\,,\\
R_{\sT M}(x^i,\dot x^j,a_k,b_l)=(x^i,b_j,-a_k,\dot x^l),\nonumber
\end{align}
is also a canonical isomorphism of double vector bundles, but this time it has an important generalization for any vector bundle $E\to M$ replacing $\sT M\to M$,
$$R_{E}:\sT^*E\to\sT^*E^*\,.$$
The latter is fundamental for the construction of geometric mechanics on Lie algebroids \cite{Grabowska:2008,Grabowska:2006}, and the proof is completely analogous to the Tulczyjew's proof for $E=\sT M$.

The importance of the Tulczyjew triple in analytical mechanics comes from the fact that it is a very elegant and useful tool in Lagrangian and Hamiltonian formalisms, including non-regular Lagrangians. Here, $M$ is a configuration manifold for a physical system, $\sT^\ast M$ is the phase space of the system, and $\sT M$ is the space of infinitesimal configurations involving positions and velocities. In this triple, the right-hand side (associated with $\sT^\ast\sT M$) is Lagrangian, the left-hand side (associated with $\sT^\ast\sT^\ast M$) is Hamiltonian, and the (generally implicit) dynamics lives {in} $\sT\sT^*M$.

In Tulczyjew's approach, this dynamics (the phase equation) is an implicit first-order differential equation on the phase space, represented by a Lagrangian submanifold $\cD$ of $\sT\sT^* M$. The latter is canonically a symplectic manifold equipped with the tangent lift $\dt(\zw_M)$ of the canonical symplectic form $\zw_M$ on $\sT^*M$. The fact that the dynamics is \emph{implicit} means that $\cD$ is generally not the image of a vector field on the phase space $\sT^*M$.

A simple way to obtain the phase dynamics is to start with a Lagrangian function $L:\sT M\to\R$, then
 $$\mathcal{D}=\alpha_M^{-1}\left(\xd L(\sT M)\right),$$
or with a Hamiltonian function $H:\sT^\ast M\to\R$, then
$$\mathcal{D}=\beta_M^{-1}(\xd H(\sT^\ast M)).$$
Of course, if the dynamics is generated by a Hamiltonian, then it is explicit, i.e., $\cD$ is the image of a Hamiltonian vector field $X_H$ on $\sT^*M$, $i_{X_H}\zw_M=-\xd H$. If the dynamics is generated by a Lagrangian, it may be implicit, since singular Lagrangians are fully accepted in this formalism. In some cases (for regular Lagrangians), the dynamics can be described simultaneously by a Lagrangian and a Hamiltonian, and we can obtain one from the other by means of the \emph{Legendre transformation}. In the case of implicit dynamics, we may still look for a Hamiltonian generating object, but then it is not a single function on the whole phase space, but a function on a submanifold or a family of functions (Morse family). The well-known formula
$$H=\langle p,v\rangle-L\,,$$
relating Hamiltonians (Morse families) with Lagrangians comes from composing generating objects for symplectic relations (see e.g. \cite{Benenti:2011}). Several examples of physical systems with various Lagrangian and Hamiltonian generating objects may be found in \cite{Tulczyjew:1999}.

Note that diagram (\ref{e:0}) is interesting also from a purely mathematical point of view. All three iterated tangent and cotangent bundles are examples of \emph{double vector bundles} which are equipped additionally with symplectic structures. The maps $\alpha_M$ and $\beta_M$ are \emph{double vector bundle morphisms} and they respect the symplectic structures; depending on conventions they are symplectomorphisms or anti-symplectomorphisms. The Hamiltonian side of the triple is built upon the canonical symplectic structure of the phase space $\sT^\ast M$, while the Lagrangian side is related to the canonical Lie algebroid structure of the tangent bundle (which again is another instance of the canonical symplectic structure on $\sT^*M$).

The origins of the Tulczyjew triple go back to the 70s' of the 20th century. The triple appeared in Tulczyjew papers \cite{Tulczyjew:1974, Tulczyjew:1976a, Tulczyjew:1976b, Tulczyjew:1977} and book \cite{Tulczyjew:1989}. For the complete theory of Legendre transformation for singular systems, we refer to \cite{Tulczyjew:1999}. The interest in the Tulczyjew triple was renewed recently after considerable developments in the theory of double vector bundles, defined first by Pradines \cite{Pradines:1974} and then developed in the context of analytical mechanics by Konieczna and Urba\'nski in \cite{Konieczna:1999}. A substantial simplification in the understanding of double vector bundles one can find in \cite{Grabowski:2009} and \cite{Grabowski:2012}.

The ideas originated in articles of Weinstein \cite{Weinstein:1996} and Libermann \cite{Libermann:1996} and concerning mechanics on Lie algebroids and Lie groupoids were developed in two distinct proposals for mechanics on Lie algebroids. One of them, by Mart\'{\i}nez, de Leon, and other members of the Spanish School, uses the concept of a Lie algebroid prolongation and follows the traditional paths of geometric mechanics \cite{deLeon:2005,Martinez:2001}. The other, by Grabowska, Grabowski, and Urba\'nski, uses the concept of the Tulczyjew triple, adapted to an algebroid as a certain double vector bundle morphism \cite{Grabowska:2006, Grabowska:2008,Grabowska:2011}. A simplified version of the Tulczyjew triple for a Lie algebroid is that for mechanics on a Lie algebra, which is deduced from the triple on a Lie group \cite{Grabowska:2016}. There are also affine triples for time-dependent mechanics \cite{Grabowska:2004}, and mechanics in Newtonian space-time \cite{Grabowska:2006a}, as well as versions of the triple for field theories, including higher-order theories \cite{Grabowska:2010,Grabowska:2012,Grabowska:2013,Grabowska:2015}.

Note that Hamiltonian formalisms are well-suited only for isolated systems with reversible dynamics. External forces can be included, especially in Lagrangian formulation, but mathematically this takes us out of the language of Lagrangian submanifolds and their generating objects. A natural question was then if it {was} possible to construct a theory with most of the geometric advantages of Hamiltonian mechanics, but with wider applications, e.g., to dissipative systems.
Hence, a major goal, at least in some areas of physics, is that of finding generalizations of Hamilton equations that apply to systems exchanging energy with the environment.
A nice model of such a theory is provided by the contact geometry. Actually, the roots of contact geometry are related to physics, e.g., to Gibbs’ work on thermodynamics, Huygens’ work on geometric optics, Hamiltonian dynamics, fluid mechanics, etc.

There is extensive literature on contact Hamiltonian systems,  e.g., \cite{Bravetti:2017,Bravetti:2017a, Ciaglia:2018,Cruz:2018,deLeon:2017,deLeon:2019,deLeon:2020,deLeon:2021,deLeon:2021a,Gaset:2020} to mention only some recent papers. They deal with various aspects of contact mechanics, including systems with constraints and time-dependent systems. Some papers include Lagrangian formalism, but mainly only in the Hamiltonian context (regular Lagrangians). There are very few publications trying to deal with singular systems. For instance, one attempt to develop a concept of a contact Tulczyjew triple can be found in \cite{Esen:2021}. Mathematically, most of these works have one disadvantage: they use only trivial (co-oriented) contact structures for which a global contact form is chosen. If we have a manifold $Q$ with a global contact form $\eta$, then we can define the Hamiltonian contact vector field $X^c_H$ for a function $H:Q \rightarrow\R$ as uniquely determined by the formulae
\begin{equation}\label{e:0.1}i_{X^c_H}\eta=-H, \qquad i_{X^c_H}d\eta=\xd H-R_\eta(H)\eta,\end{equation}
where $R_\eta$ is the {\it Reeb vector field} for $\eta$.
On the other hand, for nontrivial contact manifolds only local contact forms exist, and they are determined only up to a conformal equivalence. Even if the structure is trivializable (a global contact form exists), changing the contact form by a conformal factor results in different Hamiltonian contact vector fields associated with a given Hamiltonian so that formulae (\ref{e:0.1}) do not have geometrical character if only a global contact form is not explicitly provided. Note that examples of nontrivial contact structures are such important cases as the canonical contact structures on first jet bundles $Q=\sJ^1 L$, where $L$ is a non-trivial line bundle, or $Q=\mathbb{P}\,\sT^\ast M$, i.e., projectivized cotangent bundles for odd dimensions of $M$.

In \cite{Grabowska:2022}  we have developed a completely different picture for contact Hamiltonian formalism (and contact Hamilton-Jacobi theory) which is intrinsically geometric and valid for general contact structures. Our approach is based on a crucial observation (see \cite{Bruce:2017,Grabowski:2013}) that contact structures are, in fact, not odd-dimensional variants of symplectic structures but rather special examples of symplectic structures. There is a one-to-one canonical correspondence between contact structures (i.e., maximally non-integrable hyperplanes on odd dimensional manifolds $Q$) and certain homogeneous symplectic structures on principal bundles $\zt:P\to Q$ with the structure group $\Rt$ of multiplicative nonzero reals. This is a topologically nontrivial version of the \emph{symplectizations} of contact forms.
In this picture, contact Hamiltonians are not functions on $Q$ but 1-homogeneous functions $H$ on $P$ (or, equivalently, sections of a certain line bundle over $Q$). The Hamiltonian vector fields $X_H$ defined on $P$ are projectable to $Q$, producing the corresponding contact Hamiltonian vector fields $X^c_H$ on $Q$.

This approach can be extended, almost directly, to Jacobi structures, understood in full generality as local Lie brackets on sections of line bundles, called sometimes \emph{Kirillov brackets}. Instead of symplectic $\Rt$-bundles, we have to work with \emph{Poisson $\Rt$-bundles}, called also \emph{Kirillov manifolds} (see e.g. \cite{Bruce:2017,Grabowski:2013,Le:2018}). We can develop a Hamiltonian picture exactly like in the contact case \cite{Zapata:2020}. Of course, the Lagrangian picture in the Poisson case is not so clearly defined. A solution is to use the corresponding Lie algebroid structure on the cotangent bundle of a Poisson manifold and to develop the mechanics on Lie algebroids.

In the present paper, we extend this contact Hamiltonian formalism by a contact Lagrangian formalism represented by a contact Tulczyjew triple. Let us first stress that the original Tulczyjew triple is fundamentally related to linear symplectic geometry. It is well known that all linear symplectic structures on vector bundles over $M$ are equivalent to $(\sT^*M,\zw_M)$. The map $\beta_M$ in the triple is an instance of the canonical symplectic structure on $\sT^\ast M$. The dynamics is a Lagrangian submanifold for the canonical linear symplectic structure on $\sT\sT^*M$, and the canonical isomorphism between $\sT^*\sT^\ast M $ and $\sT^\ast\sT M$, that makes the Legendre transformation possible, is also an isomorphism of linear symplectic structures.

For contact structures, an analogous result \cite{Grabowski:2013} states that any linear contact structure on a vector bundle over $M$ is equivalent to the canonical contact structure on the vector bundle $\sJ^1L\to M$ of first jets of sections of a line bundle $L\to M$. It is then natural to expect that in the contact Tulczyjew triples the r\^ole of $\sT^*M$ with its canonical linear symplectic structure will be played by the linear contact structures on $\sJ^1L$.

A brief summary of the symplectic approach to contact structures is given in sections \ref{sec:2} and \ref{sec:3} with a focus on the canonical contact structure of $\sJ^1 L$, which is a fundamental contact structure of the contact Tulczyjew triple. In section \ref{sec:ham} we discuss Hamiltonian mechanics for contact systems, while in section \ref{sec:lag} we propose an approach to Lagrangian mechanics on contact manifolds which is different from that present in the literature. In section \ref{sec:trip} we determine the concept of a contact Tulczyjew triple and discuss the structure of all the bundles involved. Note that the contact Tulczyjew triple we propose is essentially different from that present in \cite{Esen:2021} for trivial contact structures. There are also some resemblances to affine geometric mechanics as described by Urba\'{n}ski \cite{Urbanski:2003}, but affine bundles are trivializable, so topological obstructions do not appear.

Note finally that the present paper is mainly devoted to developing the concept of contact Tulczyjew triples and basic studies on their geometry. We hope to continue working on the subject with more emphasis on applications and physics-motivated examples.

\section{Line and $\R^\ti$-principal bundles}\label{sec:2}

The content of this section is a summary of the relevant material from \cite{Grabowski:2013} and \cite{Grabowska:2022} with a few additions. Vector bundles with one-dimensional fibers will be called \emph{line bundles}. If $\rho: L\rightarrow M$ is a line bundle over a manifold $M$, then the submanifold $L^\times\subset L$ of nonzero vectors, $L^\times=L\setminus 0_M$, is canonically a principal bundle over $M$ with the structure group $(\R^\times, \cdot)$, i.e., the group of nonzero reals with multiplication. The $\R^\times$-action on $L^\times$ comes from the multiplication by reals in $L$.

Conversely, if we start with an $\Rt$-principal bundle $\pi: P\rightarrow M$, we can construct the line bundle $L_P\to M$ as the vector bundle associated with the standard $\Rt$-action on $\R$,
$$\R^\times\times\R\ni(s,r)\mapsto sr\in\R.$$
This means that elements of $L_P$ are the equivalence classes $[(p,r)]$ of pairs $(p,r)\in P\times \R$,
$$[(p,r)]=\{(h_s(p),r/s),\; s\in\R^\times\},$$
where $h$ denotes the action of $\R^\times$ on $P$,
$$h:\R^\times \times P\rightarrow P\,,\quad h(s,p)=h_s(p).$$
A fundamental observation in this context is the following.
\begin{proposition}\label{prop:6}
There is a canonical principal bundle isomorphism between $L_P^{\times}$ and $P$. In particular, the line bundle $L_P$ is trivializable if and only if the principal bundle $P$ is trivializable.
\end{proposition}
\begin{proof}
By construction, the vector bundle structure of $L_P$ comes from the vector space structure of $\R$ (the multiplication by reals completely determines the vector bundle structure \cite{Grabowski:2009}),
$$\lambda[(p,r)]=[(p,\lambda r)]\quad\text{for}\quad \zl\in\R.$$
If $s\neq 0$, then in the class $[(p,s)]$ there is a distinguished element, namely $(h_s(p),1)$. This defines an inclusion $P\ni p\mapsto [(p,1)]\in L_P$ which yields an isomorphism of the principal bundle $P$ with $L^\ti_P$,
$$h_s(p)=[(h_s(p),1)]=[(p,s)]=s[(p,1)].$$

\end{proof}

The above isomorphism actually means that by gluing $L_P$ and $P$ out of local trivial bundles we get the same transition functions. Let us observe also that elements of $L_P$ are in a one-to-one correspondence with homogeneous functions of degree $-1$ on the appropriate fiber of $P$, i.e., functions $f:P_{\pi(p)}\to\R$ such that $f(h_s(p))=s^{-1}f(p)$. The homogeneous function $f$ associated with $[(p,r)]$ is $f(h_s(p))=s^{-1}r$. Sections of $L_P$ are therefore in a one-to-one correspondence with homogeneous functions of degree $-1$ on $P$.

The dual bundle $L^\ast_P$ can be characterized as the vector bundle associated with the opposite action of $\Rt$ on $\R$,
$$\R^\times\times \R\ni(s,r)\mapsto r/s\in\R.$$
Elements of $L^\ast_P$ are then equivalence classes of pairs $(p,q)\in P\times\R$, $$[(p,q)]=\{(h_s(p),sq),\; s\in\R^\times\},$$
and it is easy to see that each equivalence class $[(p,q)]$ corresponds to a homogeneous function of degree $1$ on the fibre $P_{\pi(p)}$. Consequently, (local) sections of $L^\ast_P$ correspond to (local) $1$-homogeneous functions on $P$. If $\sigma$ is such a section, then the corresponding 1-homogeneous function we will denote $\zi_\zs$.

\begin{example}\label{ex:4}
Probably the simplest example of a line bundle that is not trivializable is that of the M\"obius band. In the trivial line bundle $\R^2\ni(x,y)\mapsto x\in\R$ we consider a linear $\Z$-action given by $$k.(x,y)=\left(x+k,(-1)^ky\right).$$
This action defines a line bundle structure on $B=\R^2/\Z$ over $S^1=\R/\Z$, where $\Z$ acts on $\R$ by translations, together with a vector bundle surjection $\R^2\to B$ covering the canonical projection $\R\to\R/\Z=S^1$. Elements of $B$ are equivalence classes $[(x,y)]$ defined by orbits of $\Z$. Suppose that there is a non-vanishing smooth section $\zs:S^1\to B$. This section has a unique smooth representative $(x,\zs_1(x))\in[0,1[\;\ti\Rt$ and a unique smooth representative $(x,\zs_2(x))\in\;]0,1]\ti\Rt$. Hence, $\zs_1$ coincides with $\zs_2$ on $]0,1[$. But $(0,y)\sim(1,-y)$, so that $\zs_2(1)=-\zs_1(0)\ne 0$; a contradiction.
Of course, in consequence, the principal bundle $B^\ti$ is not trivializable.
\end{example}

\subsection{Lifted $\Rt$-actions}
The $\Rt$-action $h:\Rt\ti P\to P$ on the $\Rt$-principal bundle $\pi:P\to M$ can be lifted to $\sT P$ and $\sT^\ast P$. In the first case, applying the tangent functor to $h_s$, we obtain $\dt h$ defined by
\be\label{tl}({\dt}h)_s=\sT h_s:\sT P\rightarrow \sT P,\ee
which is again a principal $\Rt$-action. It is easy to see that the base of the new principal bundle is $\sT P\slash \R^\times\simeq \sA P$, i.e., the so-called \emph{Atiyah algebroid} of the principal bundle $P$. Sections of this Lie algebroid are naturally identified with $\Rt$-invariant vector fields on $P$. Such vector fields on $P$ are projectable on $M$, and $X\mapsto\pi_*(X)$ represents the anchor of the Atiyah algebroid.

Starting from adapted local coordinates $(x^i,\zt)$ in $P$, we construct in the usual way the induced local coordinates $(x^i,\zt,\dot x^j,\dot\zt)$ in $\sT P$. In these coordinates, the lifted action reads
\be\label{la}
({\dt}h)_s(x^i,\zt,\dot x^j,\dot\zt)=(x^i,s\zt,\dot x^j,s\dot\zt).
\ee
Natural coordinates in $\sA P$ are then $(x^i, \dot x^i, t={\dot\zt}/{\zt})$.

Defining the phase lift is a little more tricky (there is no `cotangent functor'), however, since $h_s$  are diffeomorphisms, it can be done easily. The \emph{phase lift} $\dts h$ of the $\Rt$-action $h$ on $P$ is given by the formula
$$({\dts}h)_s(\alpha)=s\cdot(\sT h_{s^{-1}})^\ast(\alpha)$$
for $\alpha\in\sT^\ast P$.
In the induced local coordinates $(x^i,\zt,\pi_j,z)$ on $\sT^\ast P$ this action reads
\begin{equation}\label{e:17}
({\dts}h)_s(x^i,\zt,\pi_j,z)=(x^i,s\zt,s\pi_j,z).
\end{equation}
Again, the lifted action provides $\sT^\ast P$ with the structure of a principal bundle. To determine the base of this bundle, let us go back to the correspondence between 1-homogeneous functions on $P$ and sections of the bundle $L_P^\ast\rightarrow M$. This correspondence can be carried over to a correspondence between first jets of sections of $L^*_P$ and differentials of functions. Indeed, if two sections $\sigma_1$ and $\sigma_2$ of $L_P^\ast$ are such that $\sj^1\sigma_1(x)=\sj^1\sigma_2(x)$, then the two sections differ by the multiplication by a function $\lambda$ (defined on a neighbourhood of $x\in M$) such that $\lambda(x)=1$, $d\lambda(x)=0$,
$$\sigma_2(x)=\lambda(x)\sigma_1(x).$$
The corresponding homogeneous functions on $P$ are in the same relation, i.e., they differ by the multiplication by $\lambda$:
$$\imath_{\sigma_{2}}=\lambda \imath_{\sigma_{1}},$$
with some abuse of notation, since we use $\lambda$ for both, the function on $M$ and its pullback to $P$. Then, of course,
$$\xd\imath_{\sigma_{2}}=\xd(\lambda \imath_{\sigma_{1}})=\lambda \xd\imath_{\sigma_{1}}+\imath_{\sigma_{1}} \xd\lambda,$$
and since $\lambda(x)=1$ and $\xd\lambda(x)=0$, we have $$\xd\imath_{\sigma_{1}}(p)=\xd\imath_{\sigma_{2}}(p)\ \text{for any}\ p\ \text{such that}\ \zp(p)=x.$$

For a local trivialization of $P$, let us now choose local coordinates $(x^i)$ in $M$ and the adapted homogeneous coordinates $(x^i,\zt)$ on the local trivialization of $P$. The corresponding coordinates on $L^\ast_P$ will be denoted $(x^i, z)$. For $\ell\in L^\ast_P$, the vector $(x^i(\ell))$ represents the projection of $\ell$ on $M$, and $z$ is such that $\ell=[(p,z(\ell))]$ if $\zt(p)=1$.  If a section $\sigma$ of the bundle $L^\ast_P\rightarrow M$ is given in coordinates $(x^i, z)$ by a function $f$, more precisely $z(\sigma(x))=f(x^i(x))$, then the homogeneous function $\imath_{\sigma}$ on $P$ reads $\imath_{\sigma}(x^i, \zt)=\zt f(x^i)$. Both, the first jet prolongation of $\sigma$ and the differential of $\imath_\sigma$, are then determined by the values of $f$ and the partial derivatives of $f$. The differential reads
$$\xd\imath_\sigma(x^i,\zt)=f(x^i)\xd\zt+\zt\frac{\partial f}{\partial x^i}\xd x^i,$$
and $\sj^1\sigma$ is given in natural coordinates $(x^i, z, p_j)$ on $\sJ^1 L^\ast_P$ by $\big(x^i, f(x^i), \frac{\partial f}{\partial x^j}\big)$.
Note that
$$\xd\imath_\sigma(h_s(p))=({\dts}h)_s(\xd\imath_\sigma(p)).$$
The projection $\sT^\ast P\rightarrow\sJ^1L^\ast_P$ reads
$$(x^i,\zt,\pi_j,z)\longmapsto (x^i, p_j={\pi_j}/{\zt}, z).$$
The above observations about the correspondence between jets of sections and differentials of 1-homogeneous functions can be formalized as follows.

\begin{proposition}\label{prop:4} The cotangent bundle $\sT^\ast P$ equipped with the $\Rt$-action ${\dts}h$ is a $\R^\times$-principal bundle over the manifold $\sJ^1L^\ast_P$ of first jets of sections of the line bundle $L^\ast_P$.
\end{proposition}
\noindent By lifting the $\Rt$-action on $P$ to $\sT P$ and $\sT^\ast P$, we have obtained two new $\R^\times$-principal bundles, $\sT P\rightarrow \sA P$ and $\sT^\ast P\rightarrow \sJ^1L_P^\ast$. A natural question is now: what are the corresponding line bundles $L_{\sT P}$ and $L_{\sT^\ast P}$? The following proposition gives the answer.

\begin{proposition}\label{prop:5}
The line bundles $L_{\sT P}$ and $L_{\sT^\ast P}$, corresponding to the principal bundles $\sT P$ and $\sT^\ast P$ with $\Rt$-actions ${\dt}h$ and ${\dts}h$, respectively, are
$$L_{\sT P}\simeq \sA P\times_M L_P\qquad \text{and}\qquad L_{\sT^\ast P}\simeq \sJ^1L_P^\ast\times_M L_P.$$
In other words, $L_{\sT P}$ and $L_{\sT^\ast P}$ are the pull-back bundles of $L_P$ with respect to canonical projections $\sA P\to M$ and $\sJ^1L_P^\ast\to M$.
\end{proposition}
\begin{proof} By definition, points of $L_{\sT P}$ are equivalence classes $$[(v,r)]=\big\{\left({({\dt}h)_s}(v),{r}/{s}\right):\; s\in \R^\times\big\}.$$
To every $[(v,r)]$ we can associate uniquely an element of $\sA P$, namely the orbit $[v]$ of $v$ with respect to ${\dt}h$, and an element of $L_{P}$, namely $[(\tau_P(v), r)]$. This defines a vector bundle morphism
$$\Phi: L_{\sT P}\rightarrow \sA P\times_M L_P.$$
Conversely, starting from $[v]\in \sA P$ and $[(p,r)]\in L_P$ such that $\rho(\tau_P(v))=\rho(p)$, we can reconstruct the element of $L_{\sT P}$. Let $w\in[v]$ be such that $\tau_P(w)=p$. Then $[(w,r)]$ does not depend on the choice of a representative $(p,r)$ of $[(p,r)]$, and
$$\Phi([(w,r)])=([w],[(p,r)])=([v],[(p,r)]), $$
so $L_{\sT P}\simeq \sA P\times_M L_P$.
Similarly, points of $L_{\sT^\ast P}$ are equivalence classes $$[(\alpha,r)]=\big\{\big(({\dts}h)_s(\alpha),{r}/{s}\big):\; s\in \R^\times\big\}.$$
To every $[(\alpha,r)]$ we can associate uniquely an element of $\sJ^1L^\ast_P$, namely the orbit $[\alpha]$ of $\alpha$ with respect to ${\dts}h$, and an element of $L_{P}$, namely ${[(\pi_P(\alpha), r)]}$. This defines a vector bundle morphism
$$\Theta: L_{\sT^\ast P}\rightarrow \sJ^1L^\ast_P\times_M L_P.$$
Again, starting from $[\alpha]\in \sJ^1L^\ast_P$ and $[(p,r)]\in L_P$ such that $\tau(\pi_P(\alpha))=\tau(p)$, we can reconstruct the element of $L_{\sT^\ast P}$. Let $\beta\in[\alpha]$ be such that $\pi_P(\beta)=p$. Then $[(\beta,r)]$ does not depend on the choice of a representative $(p,r)$ of $[(p,r)]$, and
$$\Psi([(\beta,r)])=([\beta],[(p,r)])=([\alpha],[(p,r)]), $$
so $L_{\sT^\ast P}\simeq \sJ^1L^\ast_P\times_M L_P$.

\end{proof}
\noindent As a consequence of Proposition \ref{prop:5} we get
\begin{equation}\label{e:21}
L^\ast_{\sT P}\simeq \sA P\times_M L^\ast_P\quad\text{and}\quad L^\ast_{\sT^\ast P}\simeq \sJ^1L^\ast_P\times_M L^\ast_P.
\end{equation}
In view of Proposition \ref{prop:6}, we get also
$$
\sT P\simeq\sA P\times_M {P} \quad\text{and}\quad \sT^\ast P\simeq \sJ^1L^\ast_P\times_M {P}.
$$
The above canonical isomorphisms refer to the principal as well as to vector bundle structures. Let us recall that sections of the Atiyah algebroid $\sA P\rightarrow M$ can be interpreted as $\R^\times$-invariant vector fields on $P$. Let now $\sigma$ be a section of $L^\ast_P\rightarrow M$ and $\imath_\sigma$ be the corresponding 1-homogeneous function on $P$. If $X$ is an invariant vector field on $P$, then $X(\imath_\sigma)$ is again a 1-homogeneous function, therefore it corresponds to a section of $L^\ast_P\rightarrow M$. The vector field $X$ can be then understood as an operator acting on sections of $L^\ast_P\rightarrow M$ with values in sections of the same bundle. It is easy to see that this is in fact a first-order linear differential operator $D_X:L^\ast_P\to L^\ast_P$. This differential operator can also be viewed as a morphism of vector bundles $D_X:\sJ^1L^*_P\to L^*_P$ covering the identity on $M$. The anchor $\zp_*(X)$ can be identified with the principal symbol of $D_X$. We can formalize this observation as follows.
\begin{proposition}
There is a canonical isomorphism between sections $X$ of the Atiyah algebroid $\sA P\rightarrow M$ and first-order linear differential operators $D_X:L^\ast_P\to L^*_P$. In other words, the Atiyah algebroid $\sA P$ is exactly the Lie algebroid $DO^1(L^*_P;L^*_P)$ of first-order linear differential operators from $L^*_P$ to $L^*_P$ with the commutator bracket,
\be\label{iso1}\sA P=DO^1(L^*_P;L^*_P)=(\sJ^1L_P^*)^*\ot_ML^*_P.\ee
In other words, we have a canonical non-degenerate bilinear pairing
\be\label{pairing}
\Pi:\sA P\ti_M\sJ^1L^*_P\to L^*_P\,,\quad \zi_{\Pi(X,j^1\zs)}=X(\zi_\zs),
\ee
which in local coordinates reads
\be\label{Pi}\Pi(x^i,\dot x^j,t,p_k,z)=(x^i,\dot x^jp_j+tz).\ee
\end{proposition}
\begin{corollary} We have a canonical non-degenerate bilinear pairing
\be\label{pairing1} \sA\Lt\ti_M\sJ^1L^*\to L^*\ee
and canonical vector bundle isomorphisms
\be\label{iso} \Psi_\sA:\sA(L^*)^\ti\to(\sJ^1L)^*\ot_ML\quad \text{and}\quad
\Psi_\sJ:\sJ^1L^*\to\sA^*\Lt\ot_{M} L^*.\ee
\end{corollary}
\begin{proof}
The first isomorphism and the pairing are just  (\ref{iso1}) and (\ref{pairing}), respectively, for $P=(L^*)^\ti$, as $L^*_{(L^*)^\ti}=L$.
Tensoring both sides of (\ref{iso1}) with $L_P$, we get
$$\sA P\ot_ML_P=(\sJ^1L_P^*)^*,$$
since $L_P^*\ot_M L_P=M\ti\R$ is the trivial line bundle. Dualizing the latter for $P=\Lt$, we get
the second isomorphism.
\end{proof}

\begin{remark} Note that for any line bundle $L\to M$ there is a (non-canonical) vector bundle isomorphism $\phi:\sJ^1L^*=(\sJ^1L)^*$ covering the identity on $M$. Indeed, as $\phi$ we can take the composition of an isomorphism $\sj^1\zeta:\sJ^1L^*\to\sJ^1L$, induced by a vector bundle isomorphism $\zeta:L^*\to L$, the last one associated with a vector bundle metric on $L$, with an isomorphism $\sJ^1L\to(\sJ^1L)^*$, associated with a vector bundle metric on $\sJ^1L^*$.
\end{remark}

\subsection{Iterated tangent and cotangent bundles}
Discussing the contact version of the Tulczyjew triple we shall start from the classical version and consider the iterated tangent and cotangent bundles $\sT^\ast\sT P$, $\sT^\ast\sT^\ast P$, and $\sT\sT^\ast P$, which all are canonically (linear) symplectic manifolds.  The wings of the classical Tulczyjew triple for the manifold $P$ consist of maps $\beta_P$ and $\alpha_P$, where
$$\beta_P:\sT\sT^\ast P\ni v\longmapsto \omega_P(\cdot,v)\in \sT^\ast\sT^\ast P$$
is an instance of the canonical symplectic form $\omega_P$ on $\sT^\ast P$, and
$$\alpha_P :\sT\sT^\ast P\longrightarrow \sT^\ast\sT P$$
is the Tulczyjew isomorphism -- the dual of the canonical flip $\kappa_P:\sT\sT P\rightarrow\sT\sT P$.
It is well known that both $\beta_P$ and $\alpha_P$ are double vector bundle isomorphisms, as well as (anti)symplectomorphisms.
Our extra assumption will be that $P$ is an $\R^\times$-principal bundle. In this case, all the three manifolds $\sT^\ast\sT P$, $\sT^\ast\sT^\ast P$, and $\sT\sT^\ast P$ are equipped with the lifted $\Rt$-actions. Since $\beta_P$ and $\alpha_P$ are canonical, the following proposition is not surprising at all.
\begin{proposition}\label{prop:3} The maps $\beta_P$ and $\alpha_P$ intertwine the lifted\, $\Rt$-actions on
$\sT^\ast\sT P$, $\sT^\ast\sT^\ast P$, and $\sT\sT^\ast P$. In other words, $\beta_P$ and $\alpha_P$ are additionally principal bundle isomorphisms.
\end{proposition}
\begin{proof} First, let us observe that the canonical symplectic form $\omega_P$ on $\sT^\ast P$ is 1-homogeneous, i.e., $(({\dts}h)_s)^\ast\omega_P=s\,\omega_P.$ Indeed, for a local trivialization of $P$ and the induced adapted coordinates $(x^i,\zt,\pi_j,z)$ on $\sT^*P$ we have
$$\zw_P=\xd\pi_i\we \xd x^i+\xd z\we dt.$$
But, according to (\ref{e:17}), the coordinates $\zt,\pi_j$ are homogeneous of degree 1, and $x^i,z$ are homogeneous of degree 0. Denote for simplicity the $\Rt$-action on $\sT^\ast P$ by $k_s=({\dts}h)_s$. We should prove that $\zb_P$ intertwines the actions {$\dt k$ and $({\dts}k)$}. Indeed, we have
\begin{multline*}
\langle \zw_P(\cdot, \sT k_s(v)), w\rangle=
\zw_P(w, \sT k_s(v))=
s\,\zw_P(\sT k_{s^{-1}}(w),v)= \\
s\,\langle \zw_P(\cdot,v), \sT k_{s^{-1}}(w) \rangle=
s\,\langle (\sT k_{s^{-1}})^\ast\zw_P(\cdot,v), w \rangle=
\langle ({\dts}k)_s(\zw_P(\cdot,v)), w \rangle\,.
\end{multline*}
In the case of $\alpha_P$ we start with the observation that $\kappa_P$ intertwines the lifted $\Rt$-action on $\sT\sT P$. Indeed, the lifted action on $\sT\sT P$ is $({\dt\dt}h)_s=\sT\sT h_s$.
From the definition of the canonical flip $\zk_P$ it follows that $\sT\sT h_s\circ\kappa_P=\kappa_P\circ\sT\sT h_s$. The appropriate diagram of double vector bundle morphisms is the following
$$
\xymatrix@C+10pt{
\sT\sT P \ar[r]^{\kappa_P}\ar[d]_{\sT\sT h_s} & \sT\sT P\ar[d]^{\sT\sT h_s} \\
\sT\sT P \ar[r]^{\kappa_P} & \sT\sT P
}.
$$

\medskip\noindent Applying the duality with respect to the vector bundle structure $\sT\tau_P: \sT\sT P\rightarrow \sT P$ on the right-hand side of the diagram, and with respect to
$\tau_{\sT P}: \sT\sT P\rightarrow \sT P$ on the left-hand side, we get
$$
\xymatrix@C+10pt{
\sT^\ast\sT P & \sT\sT^\ast P\ar[l]_{\alpha_Q}\\
\sT^\ast\sT P \ar[u]^{\sT^\ast\sT h_s} & \sT\sT^\ast P\ar[l]_{\alpha_Q}
\ar[u]_{\sT\sT^\ast h_s}
},
$$

\medskip\noindent In the above, $\sT^\ast\Phi$ denotes $(\sT \Phi)^\ast$ for a diffeomorphism $\Phi$, in particular
$\sT^\ast \Phi$ is an isomorphism of the corresponding cotangent bundles over $\Phi^{-1}$. Replacing $s$ with $s^{-1}$, we get
$$\alpha_P\circ\sT\sT^\ast h_{s^{-1}}=\sT^\ast\sT h_{s^{-1}}\circ \alpha_P.$$
Multiplying by $s$ in the vector bundle structure over $\sT P$ on both sides and using the fact that $\alpha_P$ is a double vector bundle morphism, we get
$$\alpha_P\circ ({\dt}{\dts} h)_{s}=({\dts}{\dt}h)_{s}\circ \alpha_P.$$
\end{proof}
\noindent The full picture for the classical Tulczyjew triple, with the double vector bundle structures being indicated explicitly, is the following commutative diagram of surjective (double) vector bundle morphisms.
\be\xymatrix@C-20pt@R-5pt{
 &\sT^\ast \sT^* P
\ar[ddl] \ar[dr]
 &  &  & \sT \sT^\ast P \ar[lll]_{\zb_P}\ar[rrr]^{\za_P}\ar[ddl] \ar[dr]
 &  &  & \sT^\ast \sT P \ar[ddl] \ar[dr]
 & \\
 & & \sT P \ar[ddl]
 & & & \sT P\ar[rrr]\ar[lll]\ar[ddl]
 & & & \sT P\ar[ddl]
 \\
 \sT^\ast P \ar[dr]
 & & & \sT^\ast P\ar[dr]\ar[rrr]\ar[lll]
 & & & \sT^\ast P\ar[dr]
 & & \\
 & P
 & & & P \ar[rrr]\ar[lll]& & & P &
}\label{zb}\ee

\medskip\noindent Since $P$ is an $\Rt$-principal bundle, it is now clear that, in fact, all manifolds in the above diagram are canonically $\Rt$-principal bundles. By completely analogous methods as above we can prove the following extension of Proposition \ref{prop:3}.
\begin{theorem}\label{t1}
If $P$ is an $\Rt$-principal bundle, then all maps in diagram (\ref{zb})
are surjective morphisms of $\Rt$-principal bundles.
\end{theorem}

\section{Contact structures}\label{sec:3}
In the literature of the subject, a \emph{contact structure} is a \emph{contact distribution}, i.e., a maximally non-integrable distribution $C\subset\sT M$ of corank 1 on a manifold $M$ of odd dimension $2n+1$. Traditionally, the hyperplanes forming this distribution are called \emph{contact elements}. The maximal non-integrability means that the bilinear map
$$\zn_C:C\ti_MC\to\sT M/C\,,\quad \zn_C(X,Y)=\zg([X,Y])$$
is non-degenerate. Here $[X,Y]$ is the Lie bracket of vector fields belonging to $C$, and $\zg:\sT M\to\sT M/C$ is the canonical projection onto the line bundle $\sT M/C$. Note that $\zn_C$ is well defined as a bilinear map on $C$. Such a distribution is locally a kernel of a non-vanishing one form $\eta$ on $M$, and the maximal non-integrability condition is then expressed as
\begin{equation}\label{e:23}
\eta\wedge(d\eta)^n\neq 0.
\end{equation}
Such a form is called a \emph{contact form} and is not uniquely determined by $C$, since the kernels of $\eta$ and $f\eta$ are the same, provided $f$ is a non-vanishing function. A contact structure is also often understood as a manifold equipped with a global contact form $\zh$. We will call such contact structures $C=\ker(\zh)$ \emph{trivial} or \emph{co-oriented}. Note that (\ref{e:23}) implies that any trivializable contact manifold must be orientable. Of course, for a general contact structure $C$ a global contact form $\zh$ such that $C=\ker(\zh)$ need not exist.
Moreover, there is a certain inconsequence in understanding contact structures as manifolds with a chosen global contact form, since even in this case \emph{contactomorphisms} are defined as smooth maps preserving the contact form only up to a non-vanishing factor. This of course has the precise meaning that the distribution $C$, not the contact form, is preserved.

In most recent papers about contact mechanics (e.g  \cite{Bravetti:2017,Cruz:2018,deLeon:2019,deLeon:2021,deLeon:2021a,Esen:2021}) the authors work exclusively with trivial contact manifolds. We will use the more general concept of a contact manifold.
\begin{definition}
A {\it contact structure} on a manifold $M$ of odd dimension $2n+1$ is a distribution of hyperplanes $C\subset\sT M$ such that $C$ is locally a kernel of a contact form $\eta$. If the contact form $\zh$ may be chosen global, we call the contact structure \emph{trivializable} (\emph{co-orientable}); if a single global $\zh$ is chosen, the contact structure is \emph{trivial} (\emph{co-oriented}).
\end{definition}
\noindent To work with contact structures we shall use the language of \emph{symplectic $\Rt$-principal bundles}, due to the following theorem (see also \cite{Bruce:2017,Grabowska:2022}).
\begin{theorem}[\cite{Grabowski:2013}]
There is a canonical one-to-one correspondence betwe-en contact structures $C\subset\sT M$ on a manifold $M$  and symplectic $\Rt$-principal bundles over $M$, i.e., $\Rt$-principal bundles $\pi:P\to M$ equipped with a 1-homogeneous symplectic form $\zw$. In this correspondence,  the symplectic $\Rt$-principal bundle associated with $C$ is $(C^o)^\ti\subset\sT^*M$ equipped with the restriction of the canonical symplectic form $\zw_M$, where $C^o\subset\sT^*M$ is the annihilator of $C$.
\end{theorem}
Of course, $\zw$ is 1-homogeneous which means exactly that $h_s^*(\zw)=s\zw$.
\begin{definition}
An \emph{isomorphism of symplectic $\Rt$-principal bundles} $$(P_i,M_i,\zw_i),$$ $i=1,2$, is an isomorphism
\be\label{cms}
\xymatrix@C+10pt{
 P_1 \ar[r]^{\Phi}\ar[d]_{\zp_1} & P_2\ar[d]^{\zp_2} \\
M_1 \ar[r]^{\zf} & M_2}
\ee
of the $\Rt$-principal bundles $\zp_1:P_1\to M_1$ and $\zp_2:P_2\to M_2$
such that $\Phi^*(\zw_2)\\=\zw_1$. 
\end{definition}
The following is an easy exercise.
\begin{proposition}\label{cm}
If (\ref{cms}) is an isomorphism of symplectic $\Rt$-principal bundles, then $\zf$ is a contactomorphism of the corresponding contact structures on $M_1$ and $M_2$. Conversely,
if $(M_i,C_i)$, $i=1,2$, are contact manifolds and $\zf:M_1\to M_2$ is a contactomorphism, then
$(\sT\zf^{-1})^*:\sT^*M_1\to\sT^*M_2$ induces an isomorphism of symplectic $\Rt$-principal bundles
$(C_1^o)^\ti$ and $(C_2^o)^\ti$, so there is an isomorphism (\ref{cms}) of symplectic $\Rt$-bundles.
\end{proposition}

A submanifold $\cL$ in the contact manifold $M$ of dimension $(2n+1)$ we call a \emph{Legendre submanifold} if {$\dim(\cL)=n$} and vectors tangent to $\cL$ belong to $C$, $\sT\cL\subset C$. In the symplectic picture, there is a one-to-one correspondence between Legendre submanifolds in $M$ and $\Rt$-invariant Lagrange submanifolds in $P$, given by $\cL\mapsto\pi^{-1}(\cL)$.

\subsection{Canonical contact structures on jet bundles}\label{ss1}
For more details and examples of contact structures in the symplectic $\Rt$-principal bundle approach and its application to Hamiltonian mechanics and reductions we refer to \cite{Grabowski:2013}, \cite{Grabowska:2022}, and \cite{Grabowska:2023}. For the purpose of the contact Tulczyjew triple, we shall need mostly a family of canonical examples of contact manifolds, namely that of the first jet bundles of sections of line bundles. 

Let $\rho: L\rightarrow M$ be a line bundle and $\rho^\ast: L^\ast\rightarrow M$ be its dual. Then, according to Proposition \ref{prop:4}, the bundle $\sT^\ast L^\times\rightarrow \sJ^1L^\ast$ is an $\Rt$-principal bundle. Moreover, while proving Proposition \ref{prop:3} we have already checked that the canonical symplectic structure on the cotangent bundle of an $\Rt$-principal bundle is homogeneous. This means that $\sJ^1L^\ast$ carries a canonical contact structure.

The contact distribution $C\subset \sT\sJ^1L^\ast$ is spanned by vectors tangent to first jet prolongations of sections of $L^\ast$ and vectors vertical with respect to the projection $\sJ^1L^\ast\rightarrow L^\ast$. It is therefore a particular \emph{Cartan distribution} defined generally on jet bundles and playing an important r\^ole in the theory of PDEs. To see that it is indeed the contact structure, let us choose local coordinates $(x^i)$, the adapted coordinates $(x^i, z)$ in $L^*$ associated with a local trivialization, and adapted coordinates $(x^i,p_j,z)$ in $\sJ^1L^\ast$. Let us now consider a local section $\sigma$ of $L^\ast$ given in coordinates by $z=\zs(x^i)$. Then the first jet prolongation is a map
$$M\ni x\longmapsto \sj^1\sigma(x)\in \sJ^1L^\ast,$$
which in coordinates reads
$$(x^i)\longmapsto \left(x^i, \frac{\partial \zs}{\partial x^j}(x), \zs(x)\right).$$
The image of $\partial_{x^i}$ by $\sT(\sj^1\sigma)$  at $\left(x^i,p_j=\frac{\partial \zs}{\partial x^j}, z=\zs(x)\right)$ reads
$$\sT(\sj^1\sigma)(\partial_{x^i})=\partial_{x^i}+\frac{\partial^2\zs}{\partial x^i\partial x^j}\partial_{p_j}+\frac{\partial \zs}{\partial x^i}\partial z.$$
Changing $\zs$ we can get at a point $(x^i,p_j,z)$ any vector of the form $\partial_{x^i}+a_{ij}\partial_{p_j}+p_i\partial z$ with symmetric $a_{ij}\in\R$. Vectors vertical with respect to the projection $\sJ^1L^\ast\rightarrow L^\ast$ are spanned by $\partial_{p_j}$, $j=1\ldots n$. Summarizing, the distribution $C$ is spanned by
$$\big\{\partial_{x^i}+p_i\partial_z, \partial_{p_j}\big\}, \quad i,j\in\{1,\ldots,n\}.$$
It is clear that $C$ is of rank $2n$. Locally, the distribution $C$ is annihilated by $\eta=\xd z-p_i\,\xd x^i$ which is a local contact form associated with the trivialization.
Notice that changing the local linear coordinate in $L^\ast$ according to the formula  $z'=\varphi(x)z$ for some function $\varphi$, we get the transformation rules for coordinates in $\sJ^1L^\ast$ in the form
$$p'_i=\frac{\partial \varphi}{\partial x^i}(x)z+\varphi(x)p_i,\quad z'=\varphi(x)z,$$
so that the new local contact form $\xd z'-p'_i\,\xd x^i$ differs from $\xd z-p_i\,\xd x^i$ by the conformal factor $\varphi$:
$$\xd z'-p'_i\,\xd x^i=\varphi(x)(\xd z-p_i\,\xd x^i).$$

In $L$ we can use the dual coordinates $(x^i,\zt)$, i.e., the pairing between $(x,\zt)$ and $(x, z)$ at $(x^i)$ reads just $\zt z$. The same coordinates can be used in $L^\times$ with the additional condition $\zt\neq 0$. The canonical symplectic form on $\sT^\ast L^\times$ reads then in the adapted coordinates $(x^i,\zt,\pi_j,z)$
$$\omega_{L^\times}=\xd z\wedge \xd\zt+\xd\pi_i\wedge \xd x^i.$$
Then the principal bundle projection
$\sT^*L^\ti\to\sJ^1L^\ast$ reads
$$\sT^\ast L^\times\ni (x^i, \zt, \pi_j, z)\longmapsto (x^i,p_j=\pi_j/\zt,z)\in \sJ^1L^\ast.$$
We can as well use the coordinates $(x^i, \zt, p_j, z)$ on $\sT^*L^\ti$ in which the canonical symplectic form  reads
\begin{align*}
\omega_{L^\times}&=\xd z\wedge\xd\zt+\zt \xd p_j\wedge \xd x^j +p_i\,\xd\zt\wedge\xd x^i\\
&=
(\xd z-p_i\,\xd x^i)\wedge\xd\zt +\zt\,\xd p_j\wedge \xd x^j=\zh\we \xd\zt-\zt\,\xd\zh.
\end{align*}
The form $\omega_{L^\times}$ is explicitly 1-homogeneous. Contracting it with the `Euler vector field' $\n$ associated with the $\Rt$-action (the opposite of the fundamental vector field), which is homogeneous of degree 0 and in coordinates reads $\n=\zt\,\partial_\zt$, we get a 1-homogenous one form
$$\zvy=i_\n\omega_{L^\times}=-\zt(\xd z-p_i\,\xd x^i)=-\zt\,\zh,$$
which is a potential for the symplectic form,
$$\omega_{L^\times}=\xd \zvy.$$
This form is semibasic with respect to the projection onto $\sJ^1L^\ast$. We can recover the local contact form $\zh$ as
$$\zh=\xd z-p_i\,\xd x^i=-{\zt}^{-1}\zvy.$$
Note that $\zvy$ is a geometric object associated with the symplectic $\Rt$-principal bundle structure and independent of the choice of trivialization and coordinates. On the other hand, $\eta=-\zvy/\zt$ involves the coordinate $\zt$, so depends on the local trivialization.
\begin{proposition}\label{triv}
The canonical contact structure on $\sJ^1L^*$ is trivializable if and only if $L$ (thus $L^*$) is trivializable.
\end{proposition}
\begin{proof}
The `if' part is obvious. Suppose that there is a global contact form $\zh$ for the canonical contact structure on $\sJ^1L^*$. We know that for local coordinates $(x^i,p_j,z)$ on $\sJ^1L^*$, associated with a local trivialization of $L^*$, a contact form for the canonical structure is $\xd z-p_i\,\xd x^i$, so
$$\zh(x^i,p_j,z))=F(x^i,p_j,z)\left(\xd z-p_i\,\xd x^i\right)$$ for a local non-vanishing function $F$ on $\sJ^1L^*$. Restricting $\zh$ to the zero-section, we get $\zh(x,0,0)=\xd(f(x)z)(x,0,0)$, where $f(x)=F(x,0,0)$. Here, $f(x)z$ is a local linear function on $L^*$,  i.e., of the form $\zi_\zs$ (more precisely, the pull-back of $\zi_\zs$ with respect to the canonical projection $\sJ^1L^*\to L^*$), where $\zs$ is a local section of $L$. Two different linear functions cannot differ by a nonzero constant, so $\zs$ is determined uniquely.  But $\zh$ is global, so $\zh(x,0,0)=\xd(\zi_\zs)(x,0)$ defines uniquely a section $\zs$ of $L$. If $\zs(x)=0$, then $\xd(\zi_\zs)(x,0)=0$, therefore $\zs$ is non-vanishing, since $\zh$ is non-vanishing.
\end{proof}
\begin{example}
In the case of the trivial line bundle $L=M\times \R$, both the dual bundle $L^\ast\simeq M\times \R^\ast$ and the first jet bundle $\sJ^1L^\ast\simeq \sT^\ast M\times\R^\ast$ are also trivial. All the above calculations have then a geometric character because we have distinguished global coordinates $\zt$ in $L$ and $z$ in $L^\ast$. There is then a distinguished global contact form on $\sJ^1L^\ast\simeq \sT^\ast M\times\R^\ast$, namely $\eta=\xd z-\theta_M$, where $\theta_M=p_i\,\xd q^i$ is the canonical Liouville 1-form on $\sT^\ast M$.
\end{example}
\begin{example}\label{AB}
The M\"obius band, as the line bundle $B\to S^1$ of Example \ref{ex:4}, can be described by two charts. If $x\notin \Z$ then in every equivalence class $[(x,y)]$ there is just one representative with $x\in]0,1[$. For further purposes, it will be convenient to reparameterize $x\mapsto \pi x$. We take then $\mathcal{O}=\{[(x,y)]: x\in]0,\pi[\}$. Similarly, if $x\slash \pi\neq\frac{2k+1}{2}$, then in every equivalence class $[(x,y)]$ there is just one representative with $x\in\big]\frac{\pi}{2},\frac{3\pi}{2}\big[$, we take then $\mathcal{U}=\{[(x,y)]: x\in\big]\frac{\pi}{2},\frac{3\pi}{2}\big[\}$. Both $\mathcal{O}$ and $\mathcal{U}$ are open and we have $\mathcal{O}\cup\mathcal{U}=B$. We can use in $B$ the following two charts:
$$\varphi:\mathcal{O}\rightarrow \R^2\,, \quad\varphi([(x,y)])=(x, y)\ \text{for}\ x\in \big]0,\pi\big[$$ and
$$\psi:\mathcal{U}\rightarrow \R^2\,,\quad\psi([(x',y')])=(x', y')\ \text{for}\ x'\in\Big]\frac{\pi}{2},\frac{3\pi}{2}\Big[.$$
The transition function
$$\psi\circ\varphi^{-1}(x,y)=(x,y)\ \text{for}\ x\in\Big]\frac{\pi}{2},\pi\Big[$$
and
$$\psi\circ\varphi^{-1}(x,y)=(x+\pi,-y)\ \text{for}\ x\in\Big]0,\frac{\pi}{2}\Big[$$
is linear in the second coordinate, therefore it defines a line bundle structure on $B$.
The first jet bundle $\sJ^1B$ of the M\"obius band can be, again, defined by two charts. Denoting with $\sj^1\varrho$ the projection $\sj^1 \varrho: \sJ^1B\rightarrow B$, as the domains of the two charts we take
$\bar{\mathcal{O}}=(\sj^1\varrho)^{-1}(\mathcal{O})$  and  $\bar{\mathcal{U}}=(\sj^1\varrho)^{-1}(\mathcal{U})$. The adapted coordinates in $\bar{\mathcal{O}}$ for the map $\bar{\varphi}$ are
$$(x, p, \zt)\in\,]0,\pi[\times\R\times \R,$$
while the adapted coordinates in $\bar{\mathcal{U}}$ for the map $\bar{\psi}$ are
\be\label{tr1}(x',p',\zt')\in\,\Big]\frac{\pi}{2},\frac{3\pi}{2}\Big[\times\R\times \R,\ee
with the transformation rule
\be\label{tr2}\bar{\psi}\circ\bar{\varphi}^{-1}(x,p,\zt)=(x,p,\zt)\quad\text{if}\quad x\in\Big]
\frac{\pi}{2}, \pi\Big[\ee
and
$$ \bar{\psi}\circ\bar{\varphi}^{-1}(x,p,\zt)=(x+\pi,-p,-\zt)
\quad\text{if}\quad x\in\Big]0,\frac{\pi}{2}\Big[.$$
The coordinates $p$ and $p'$ change sign in the same way as $y$ and $y'$ because if a section $\sigma$ is given in chart $\mathcal{O}$ as $x\mapsto (x, y(x))$ then the first jet $\sj^1\sigma$ in chart $\bar{\mathcal{O}}$ has coordinates $p(\sj^1\sigma(x))=\frac{\pa y}{\pa x}$ and $\tau(\sj^1\sigma(x))=y(x)$.
The contact distribution over $\bar{\mathcal{O}}$ is spanned by $(\partial_{x}+p\partial_\zt, \partial_{p})$, and over $\bar{\mathcal{U}}$ by
$(\partial_{x'}+p'\partial_{\zt'}, \partial_{p'})$. It is easy to see that for points for which $x\in\big]0,\frac{\pi}{2}\big[$ we have
$\partial_{p'}=-\partial_{p}$, but of course this yields a well-defined global distribution $C$. The corresponding local contact form forms on $\bar{\mathcal{O}}$ are equivalent to
$\eta_\mathcal{O}=\xd\zt-p\,\xd x$, and the local contact forms on $\bar{\mathcal{U}}$ are equivalent to $\eta_\mathcal{U}=\xd\zt'-p'\,\xd x'$. Again, if $x\in\big]0,\frac{\pi}{2}\big[$, then we have
$\eta_\mathcal{O}=-\eta_\mathcal{U}$, while for $x\in\big]\frac{\pi}{2},\pi\big[$ we have $\eta_\mathcal{O}=\eta_\mathcal{U}$. Therefore we cannot agree local contact forms for these two charts, so the contact distribution $C$ is well defined, but there is no global contact form for $C$. Of course, it follows also from the general fact (see Proposition \ref{triv}).

Note that the vector bundle $\sJ^1 B\rightarrow S^1$ is trivializable. Let us consider the following two sections of the bundle $B\rightarrow S^1$, written in coordinates from the chart $(\mathcal{O},\varphi)$,
$$\sigma_1(x)=(x,\sin(x)), \qquad \sigma_2(x)=(x, \cos(x)).$$
It is easy to check that they are well defined on the whole $S^1$. Using the chart $(\bar{\mathcal{O}},\bar\varphi)$, we can write the first jet prolongations of $\sigma_1$ and $\sigma_2$ as
$$\sj^1\sigma_1(x)=(\xi,\sin(x), \cos(x)), \qquad \sj^1\sigma_2(x)=(\xi, \cos(x), -\sin(x)).$$
The two sections $\sj^1\sigma_1$ and $\sj^1\sigma_2$ are global, non-vanishing, and linearly independent sections of $\sJ^1 B\rightarrow S^1$. This shows that $\sJ^1L$ may be trivializable even for a non-trivializable $L$.
\end{example}

An interesting observation is that the tangent lifts of $\Rt$-principal bundles (cf. (\ref{tl})), combined with the tangent lifts of symplectic structures (see, e.g., \cite{Grabowski:1995}), define the \emph{tangent lifts of contact structures} as follows.

\begin{theorem} If $\zp:P\to M$ is a symplectic $\Rt$-principal bundle, with  an $\Rt$-action $h$ and a 1-homogeneous symplectic form $\zw$, then $\sT P$ is canonically a symplectic $\Rt$-principal bundle over the Atiyah algebroid $\sA P$, equipped with the lifted action $\dt h$ and the lifted symplectic structure $\dt\zw$.
\end{theorem}
\begin{proof}
We have to check that $\dt\zw$ is 1-homogeneous with respect to $\dt h$. Let $(x^i,\zt)$ be local coordinates on $P$, associated with a local trivialization of the principal bundle, and let $\zh$ be a local contact form on $M$ representing the contact structure on $M$. Then the symplectic form $\zw$ can be written as (see \cite{Bruce:2017,Grabowska:2022})
$$\zw(x,\zt)=\zh(x)\we\xd\zt -\zt\cdot\xd\zh(x).$$
In the adapted coordinates $(x^i,\zt,\dot x^j,\dot\zt)$ on $\sT P$, the lifted $\Rt$-action reads (see (\ref{la}))
$$(\dt h)_s(x^i,\zt,\dot x^j,\dot\zt)=(x^i,s\zt,\dot x^j,s\dot\zt),$$
and the lifted symplectic form (see \cite{Grabowski:1995})
$$\dt\zw=\zh\we \xd\dot\zt+\dt\zh\we \xd\zt-\dot\zt\cdot \xd\zh-\zt\cdot\dt\zh.$$
Since $\dt\zh$ is a 1-form on $\sT M$, it involves only coordinates $(x^i,\dot x^j)$, so it is homogeneous of degree 0. This immediately implies that $\dt\zw$ is homogeneous of degree 1, since
$(\zt,\dot\zt)$ are homogeneous of degree $1$.

\end{proof}

The above proposition implies that there is a canonical contact structure on the Atiyah algebroid $\sA P$. This contact structure we will call the \emph{principal lift} of the contact structure on $M$. This lift is probably new in the literature and difficult to express in the traditional language of contact geometry. It is a contact analog of the tangent lift of a symplectic structure.

\section{Contact Hamiltonian mechanics}\label{sec:ham}
The term `Hamiltonian mechanics' is usually connected with the symplectic structure of a phase space. In most cases, the phase space is the cotangent bundle $\sT^\ast M$, and the symplectic structure is the canonical form $\omega_M$. With a function $H$ on the phase space (Hamiltonian function) we associate the Hamiltonian vector field $X_H$ by
$$
\xd H=\omega_M(\cdot, X_H).
$$
Integral curves of the Hamiltonian vector field are phase trajectories of the mechanical system. In the Tulczyjew's formulation of mechanics, we replace $\zw_M$ with the map
$$\beta_M:\sT\sT^\ast M\rightarrow \sT^\ast\sT^\ast M,$$
which is determined by $\omega_M$,
$$\sT\sT^\ast M\ni w\longmapsto \beta_M(w)=\omega_M(\cdot,w)\in\sT^\ast\sT^\ast M.$$
A mechanical system is then described by a subset $\mathcal{D}$ of $\sT\sT^\ast M$, called the \emph{dynamics}, which is understood as an implicit first-order differential equation. The phase trajectories are in this case curves in $\sT^*M$ such that their tangent prolongations belong to $\mathcal{D}$. In the case of Hamiltonian mechanics with the Hamiltonian $H$, the dynamics is just the image of the Hamiltonian vector field. In terms of $\beta_M$ we can write
$$\mathcal{D}=\beta_M^{-1}\left(\xd H(\sT^\ast M)\right)=X_H(\sT^\ast M).$$
Passing from $\omega_M$ to $\beta_M$ is only a slight change of interpretation, but an important one, because it allows for a Hamiltonian description of more general mechanical systems for which there is no Hamiltonian understood as a single function on the phase space.

In the following, we shall construct the contact analog of the map $\beta_M$ in the case when the phase space carries a contact structure. More precisely, we replace the linear symplectic manifold $\sT^\ast M$ with the linear contact manifold $\sJ^1 L^\ast$ for some line bundle $L$ over M. Note that all linear contact manifolds are equivalent to one of $\sJ^1 L^\ast$ \cite{Grabowski:2013} like all linear symplectic manifolds are equivalent to $\sT^*M$. This choice of contact manifolds is important for our construction of contact Tulczyjew triples, but we can actually construct the contact Hamiltonian mechanics starting from an arbitrary contact structure, as we have done in \cite{Grabowska:2022}. The point there, which is different from most approaches to contact Hamiltonian mechanics in the literature, is that Hamiltonians are not functions on the contact manifold $Q$ itself, but 1-homogeneous functions on the corresponding symplectic $\Rt$-principal bundle over $Q$.

\subsection{Contact Hamiltonian vector field}\label{s4.1}
Let us recall that the symplectic $\Rt$-principal bundle associated with the contact manifold $\sJ^1L^\ast$ is $\sT^*L^\ti$. Before we pass to the Hamiltonian side of the contact Tulczyjew triple, let us fix a line bundle $L\rightarrow M$ and calculate in coordinates the Hamiltonian vector field associated with a 1-homogeneous Hamiltonian on $\sT^\ast L^\times$. Recall that $\sT^\ast L^\times\rightarrow \sJ^1L^\ast$ is canonically a symplectic $\R^\times$-principal bundle, with the $\Rt$-action being the phase lift of the obvious action $k$ on $L^\ti$ (from now on, with $\mathcal{h}$ we will denote contact Hamiltonians) and the canonical symplectic form $\zw_{L^\ti}$ on $\sT^*\Lt$. Starting with coordinates $(x^i,\zt)$ on $\Lt$, associated with a local trivialization, the action ${\dts}k$ in the adapted coordinates reads
$$(\dts k)_s(x^i,\zt,\pi_j,z)=(x^i,s\zt,s\pi_j,z).$$
Any 1-homogeneous function on $\sT^\ast L^\times$ is then of the form
$$H(x^i,\zt,\pi_j,z)=\zt\cdot \mathcal{h}\left(x^i, \frac{\pi_j}{\zt},z\right),$$
where $h$ is a function of coordinates $(x^i, p_j, z)$ on $\sJ^1L^*$.
Using the canonical symplectic form $\omega_{L^\times}=\xd\pi_i\wedge \xd x^i+\xd z\wedge \xd\zt$, we get the Hamiltonian vector field in the form
\begin{equation}\label{e:1}
X_H=\frac{\partial \mathcal{h}}{\partial p_i}\partial_{x^i}+\zt\frac{\partial \mathcal{h}}{\partial z}\partial_\zt-\zt\frac{\partial \mathcal{h}}{\partial x^i}\partial_{\pi_j}+\left(\frac{\pi_k}{\zt}\frac{\partial \mathcal{h}}{\partial p_k}-\mathcal{h}\right)\partial_z.
\end{equation}
From (\ref{e:1}) we immediately see that $X_H$ is projectable on $\sJ^1L^\ast$.
The projected vector field reads
\begin{equation}\label{e:2}
X_h=\frac{\partial \mathcal{h}}{\partial p_i}\partial_{x^i}-\left(\frac{\partial \mathcal{h}}{\partial x^j}+p_j\frac{\partial \mathcal{h}}{\partial z}\right)\partial_{p_j}+\left(p_k\frac{\partial \mathcal{h}}{\partial p_k}-\mathcal{h}\right)\partial_z,
\end{equation}
so we have recovered in local coordinates the formula which is well known in the literature on contact Hamiltonians. This time, however, this formula has an unambiguous global meaning valid also for contact structures that are nontrivial.

Every 1-homogeneous Hamiltonian $H$ on $\sT^*\Lt$ corresponds to a certain section of the line bundle $$L_{\sT^\ast L^\times}\simeq \sJ^1L^\ast\times_M L\raa\sJ^1L^\ast,$$
which in the adapted coordinates $(x^i, p_j, z, Z)$ is given by the condition $Z=h(x^i,p_j, z)$. Formula (\ref{e:2}) is therefore the expression for the contact Hamiltonian vector field generated by the section $\mathcal{h}$ of the line bundle $L^*_{\sT^\ast L^\times}$. Taking into account the identification $L_{\sT^\ast L^\times}\simeq \sJ^1L^\ast\times_M L$, we can write the contact Hamiltonian $\mathcal{h}$ as a map
$$\mathcal{h}: \sJ^1L^\ast\longrightarrow L^\ast$$
covering the identity on $M$. Using the same letter $\mathcal{h}$ to describe the section may be confusing, however, this is what we usually do with ordinary functions on a manifold -- we write $f$ for a function $f:N\rightarrow\R$ as well as its expression $f(y^a)$ in coordinates. The following is a useful observation we will use while considering contact Legendre maps.
\begin{proposition}
The line bundle $L_{\sT^*\Lt}$ is trivializable if and only if $L$ is trivializable.
\end{proposition}
\begin{proof}
Suppose that $L_{\sT^*\Lt}$ is trivializable and let $\mathcal{h}: \sJ^1L^\ast\longrightarrow L^\ast$ represents a non-vanishing section. Since $L_{\sT^*\Lt}$ is a vector bundle over $M$, we can identify points $x\in M$ with the corresponding zero vectors $0_x$ in $L_{\sT^*\Lt}$. Now, we define a non-vanishing section $\zs$ of $L^*$ by $\zs(x)=\mathcal{h}(0_x)$. Of course, $L^*$ is trivializable if and only if $L$ is trivializable. The converse is trivial.

\end{proof}
\subsection{The Hamiltonian side of the contact Tulczyjew triples}
Let us now pass to the Hamiltonian side of the contact Tulczyjew triple. It can be obtained from the classical map $\beta_{L^\times}$ by means of Theorem \ref{t1}.

\begin{theorem}
The Tulczyjew double vector bundle isomorphism
\be\label{b1}\xymatrix@C-20pt@R-10pt{
& \sT^\ast\sT^\ast \Lt  \ar[dr] \ar[ddl]
 & & & \sT\sT^\ast \Lt\ar[dr]\ar[ddl]\ar[lll]_{\beta_\Lt}&  \\
 & & \sT \Lt\ar[ddl]
 & & & \sT \Lt \ar[ddl]\ar[lll]\\
 \sT^\ast \Lt\ar[dr]
 & & & \sT^\ast \Lt\ar[dr]\ar[lll] & &  \\
 & \Lt& & & \Lt\ar[lll] &
}\ee
induces canonically the double vector bundle isomorphism described by the following commutative diagram of surjective vector bundle morphisms.
\be\label{b2}\xymatrix@C-18pt@R-10pt{
& \sJ^1 L^\ast_{\sT^\ast L^\times}  \ar[dr] \ar[ddl]
 & & & \sA\sT^\ast L^\times\ar[dr]\ar[ddl]\ar[lll]_{\beta_0}&  \\
 & & \sA \Lt\ar[ddl]
 & & & \sA \Lt \ar[ddl]\ar[lll]\\
 \sJ^1L^*\ar[dr]
 & & & \sJ^1L^*\ar[dr]\ar[lll] & &  \\
 & M& & & M\ar[lll] &
}\ee
Moreover, $\sJ^1 L^\ast_{\sT^\ast L^\times}$ and $\sA\sT^\ast L^\times$ possess canonical contact structures and the map $\zb_0$ is additionally a contactomorphism.
\end{theorem}

\medskip\noindent\begin{proof} According to Theorem \ref{t1}, all maps in diagram (\ref{b1}) are surjective morphisms of $\Rt$-principal bundles. We obtain a diagram (\ref{b2}) just passing to the maps induced on the base manifolds of these principal bundles.
For contact structures, we refer to canonical contact structures on jet bundles and Atiyah algebroids described in Section \ref{ss1} and Proposition \ref{cm}.

\end{proof}
\begin{remark}
Since $\zb_0$ is a canonical diffeomorphism, the Atiyah algebroid $\sA\sT^\ast L^\times$ has a canonical structure of a first jet bundle, and the first jet bundle $\sJ^1 L^\ast_{\sT^\ast L^\times}$ has a canonical structure of a Lie algebroid. It would be interesting to describe these structures in intrinsic terms, not referring to the map $\zb_0$. We are guessing that this can be done by means of certain natural Jacobi brackets.
\end{remark}
\noindent We will use the following local coordinates:
\begin{itemize}
\item local coordinates $(x^i)$ on $M$;
\item the corresponding affine coordinates $(x^i,\zt)$ on $\Lt$ (or $L$);
\item the dual coordinates $(x^i,z)$ on $L^*$;
\item the adapted local coordinates $(x^i,\zt,\zp_j,z)$ on $\sT^*\Lt$;
\item the adapted local coordinates $(x^i,\zt,\zp_j,z,\dot x^k,\dot\zt,\dot\zp_l,\dot z)$ on $\sT\sT^*\Lt$;
\item the adapted local coordinates $(x^i,\zt,\zp_j,z,P_k,P_\zt,P^l,P_z)$ on $\sT^*\sT^*\Lt$.
\end{itemize}

Since both side bundles $\sT^\ast\sT^\ast L^\times$ and $\sT\sT^* L^\times$ of the classical Tulczyjew triple are canonically isomorphic double vector bundles with compatible  structures of  $\R^\times$-principal bundles, the quotiening by the group actions produces again a pair of isomorphic double vector bundles $\sJ^1L^*_{\sT^*\Lt}$ and $\sA\sT^*\Lt$, with the induced coordinates
$(x^i,p_j,z,Z_k,Z^l,Z_z,Z)$ and $(x^i,p_j,z,\dot x^k,\dot p_l,\dot z,t)$, respectively. Here, the coordinates on $L^*_{\sT^*\Lt}=\sJ^1L^*\ti_ML^*$ are $(x^i,p_j,z,Z)$. The projection $\sT^*\sT^*\Lt\to\sJ^1L^*_{\sT^*\Lt}$ is given by
$$
\left(x^i,p_j,z,Z_k,Z^l,Z_z,Z\right)=\left(x^i,\frac{\zp_j}{\zt},z,
\frac{P_k}{\zt},\,P^l,\frac{P_z}{\zt},P_\zt+\frac{\zp_rP^r}{\zt}\right),
$$
and the projection $\sT\sT^* L^\times\to\sA\sT^*\Lt$ by
$$
\left(x^i,p_j,z,\dot x^k,\dot p_l,\dot z,t\right)=\left(x^i,\frac{\zp_j}{\zt},z,\dot x^k,\frac{\dot\zp_l}{\zt}-\frac{\zp_l\dot\zt}{\zt^2},\dot z,\frac{\dot\zt}{\zt}\right).
$$
The isomorphism $\zb_0$ is now
$$
\left(x^i,p_j,z,Z_k,Z^l,Z_z,Z\right)\circ\zb_0=\left(x^i,p_j,z,\dot x^k,-\dot p_l-p_lt,t,
p_r\dot x^r-\dot z\right).
$$

The induced double vector bundle structures are, respectively,
$$
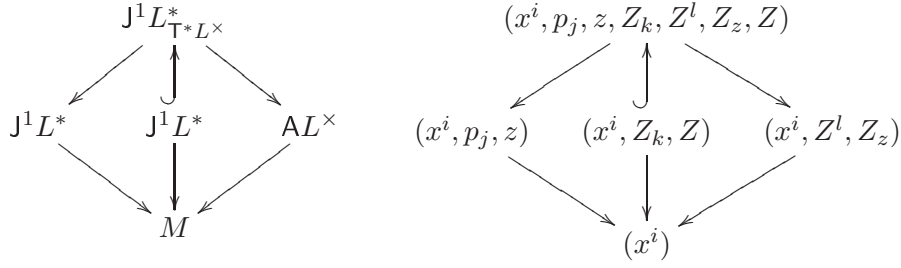
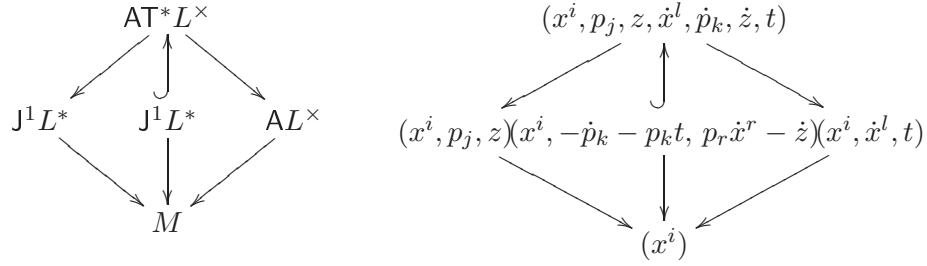
\xymatrix@C-10pt{
 & \sJ^1L^\ast_{\sT^\ast L^\times}\ar[dl]\ar[dr] & \\
\sJ^1L^\ast\ar[dr] & \sJ^1L^\ast\ar[d]\ar@{ (->}[u] & \sA L^\times\ar[dl] \\
 & M &
}\qquad
\xymatrix@C-40pt{
 & (x^i,p_j,z,Z_k, Z^l, Z_z,Z)\ar[dl]\ar[dr] & \\
(x^i,p_j, z)\ar[dr] & (x^i, Z_k, Z)\ar[d]\ar@{ (->}[u] & (x^i, Z^l, Z_z)\ar[dl] \\
 & (x^i) &
}
$$
and
$$
\xymatrix@C-10pt{
 & \sA\sT^*\Lt\ar[dl]\ar[dr] & \\
\sJ^1L^\ast\ar[dr] & \sJ^1L^\ast\ar[d]\ar@{ (->}[u] & \sA L^\times\ar[dl] \\
 & M &
}\qquad
\xymatrix@C-32pt{
 & (x^i,p_j,z,\dot x^l,\dot p_k,\dot z,t)\ar[dl]\ar[dr] & \\
(x^i,p_j, z)\ar[dr] & (x^i,-\dot p_k-p_kt,\, p_r\dot x^r-\dot z)\ar[d]\ar@{ (->}[u] & (x^i, \dot x^l, t)\ar[dl] \\
 & (x^i) &
}
$$

\noindent To have a dynamical part in the picture, we have to modify slightly commutative diagram \ref{b2}. For, we shall use the anchor map of the Atiyah algebroid,
$${\cA}: \sA\sT^\ast L^\times\rightarrow\sT\sJ^1L^\ast.$$
The full commutative diagram associated with the anchor is
\be\label{cA}\xymatrix@C-18pt@R-10pt{
& \sA\sT^\ast L^\times \ar[rrr]^\cA \ar[dr] \ar[ddl]
 & & & \sT\sJ^1L^*\ar[dr]\ar[ddl]&  \\
 & & \sA \Lt\ar[rrr]\ar[ddl]
 & & & \sT M \ar[ddl]\\
 \sJ^1L^*\ar[dr]\ar[rrr]
 & & & \sJ^1L^*\ar[dr] & &  \\
 & M\ar[rrr]& & & M &
}\ee

\medskip\noindent The map $\sA \Lt\to\sT M$ is of course the anchor of the Atiyah algebroid $\sA\Lt$. In coordinates, the anchors are just forgetting the coordinate $t$. The following is now straightforward.
\begin{proposition}
The composition of maps $\zb=\cA\circ\zb_0^{-1}$ defines a morphism of double vector bundles
$$\beta:\sJ^1 L_{\sT^\ast L^\times}\rightarrow \sT\sJ^1 {L^\ast},$$
with the corresponding commutative diagram of vector bundle morphisms
\be\label{Hs}\xymatrix@C-18pt@R-10pt{
& \sJ^1 L^\ast_{\sT^\ast L^\times} \ar[rrr]^\zb \ar[dr] \ar[ddl]
 & & & \sT\sJ^1L^*\ar[dr]\ar[ddl]&  \\
 & & \sA \Lt\ar[rrr]\ar[ddl]
 & & & \sT M \ar[ddl]\\
 \sJ^1L^*\ar[dr]\ar[rrr]
 & & & \sJ^1L^*\ar[dr] & &  \\
 & M\ar[rrr]& & & M &
}\ee
\end{proposition}

\medskip\noindent The above diagram represents exactly the Hamiltonian part of the (dynamical) contact Tulczyjew triple. Note that, in contrary to $\beta_0$, the map
$\beta$ is no longer a diffeomorphism (the dimension of $\sT\sJ^1 {L^\ast}$ is by 1 smaller than the dimension of $\sJ^1 L_{\sT^\ast L^\times}$) and acts, in a sense, in the direction opposite to that of $\beta_0$; however, it serves the same purposes as $\beta_M$ in the case of Hamiltonian mechanics. Using $\beta$, we can write the contact Hamiltonian vector field generated by a contact Hamiltonian, i.e., a section
$$\mathcal{h}:\sJ^1L^\ast\to L_{\sT^\ast L^\times}\simeq \sJ^1L^\ast\times_M L,$$
(often viewed as a map $\mathcal{h}:\sJ^1L^\ast\to L$ covering the identity on $M$) as
$$X_\mathcal{h}=\beta\circ\sj^1 \mathcal{h},$$
which looks very similar to the definition of the symplectic Hamiltonian vector field $X_H=\beta^{-1}_M\circ \xd H$ for $H:\sT^\ast M\rightarrow \R$. In coordinates,
$$
(\,x^i,\,p_j,\,z,\,\dot x^k,\,\dot p_l,\,\dot z\,)\circ\zb=(x^i,\, p_j,\,z,\, Z^k,\, -Z_l-Z_zp_l,\, p_rZ^r-Z).
$$

\medskip\noindent
The map $\beta_{L^\times}$ is an anti-symplectomorphism, as well as a principal bundle isomorphism, which means that it maps invariant Lagrangian submanifolds to invariant Lagrangian submanifolds. Since Legendrian submanifolds of contact manifolds are just images of invariant Lagrangian submanifolds in the corresponding symplectic $\Rt$-principal bundles under the principal bundle projections, we infer that $\beta_0$ maps Legendrian submanifolds to Legendrian submanifolds. The map $\beta$ can then serve as a tool for obtaining dynamics from more general Legendrian submanifolds than just the images of the first jet prolongations $\sj^1 \mathcal{h}$ of contact Hamiltonians $\mathcal{h}$.

\begin{example}\label{ex:1} Let us construct the map $\beta$ from the above proposition for the trivial line bundle $L=M\times\R$. For that, we have to examine all the spaces and maps from the diagram (\ref{Hs}), assuming the triviality of $L$.

Points of the trivial line bundle $L=M\times\R$ will be denoted by $(x,\zt)$, with $x\in M$ and $\zt\in\R$. We shall use the same notation for elements $(x,\zt)$ of $L^\times=M\times \R^\times$, with the additional condition $\zt\neq 0$. The line bundle $L^\ast$ can be identified with $M\times \R^\ast$ with elements $(x,z)$. We have then
$$\sT^\ast L^\times\simeq \sT^\ast M\times\R^\times\times\R^\ast,\qquad \sJ^1L_P\simeq \sT^\ast M\times\R^\ast,$$
and the projection reads
\begin{equation}\label{e:5}
\sT^\ast M\times\R^\times\times\R^\ast\ni(\pi,\zt,z)\longmapsto \left(\frac{\pi}{\zt}, z\right)\in \sT^\ast M\times\R^\ast.
\end{equation}
By definition, the Atiyah algebroid of the principal bundle $\sT^\ast L^\times\rightarrow \sJ^1L^\ast$ is $\sT\sT^\ast L^\times\slash \R^\times$. We identify $\sT\sT^\ast L^\times$ with
$\sT\sT^\ast M\times \R^\times \times \R^\ast\times \R\times\R^\ast$ with elements $(w,t,z,\dot t, \dot z)$. The lifted $\Rt$-action reads then
$$({\dt}{\dts}k)_s(w,\zt,z,\dot \zt, \dot z)=(s\odot w,s\zt,z,s\dot \zt,\dot z),$$
where $\odot$ stands for the multiplication by reals in the vector bundle $\sT\pi_M: \sT\sT^\ast M\rightarrow \sT M$. We have then the following identifications and  projections:
$$\sT\sT^\ast L^\times\simeq\sT\sT^\ast M\times \R^\times \times \R^\ast\times \R\times\R^\ast, \qquad  \sA\sT^\ast L^\times\simeq\sT\sT^\ast M\times\R^\ast\times \R\times\R^\ast,$$
$$(w,t,z,\dot t, \dot z)\longmapsto \left(\frac{1}{t}\odot w,\, z,\,\frac{\dot t}{t},\, \dot z\right).$$
The anchor map ${\cA}$ for $\sA\sT^\ast L^\times$ can be read from the tangent of the map given by equation (\ref{e:5}). The tangent map is
$$(w,\zt,z,\dot \zt, \dot z)\longmapsto \left(\frac{1}{\zt}\odot w\,\oplus\,\frac{-\pi\dot \zt}{\zt^2},\, z,\, \dot z\right),$$
where  $\pi=\tau_{\sT^\ast M}(w)$ and $\oplus$ stands for the operation of adding an element of the core (in this case $\sT^\ast M$) to the element of the double vector bundle $\sT\sT^\ast M$. The anchor then reads
$${\cA}:\sA\sT^\ast L^\times\simeq\sT\sT^\ast M\times\R^\ast\times \R\times\R^\ast\longrightarrow \sT\sT^\ast M\times\R^\ast\times \R^\ast\simeq \sT\sJ^1 L^\ast,$$
\begin{equation}\label{e:12}
(w,z,\tau,\dot z)\longmapsto (w\oplus (-\tau\pi), z,\dot z).
\end{equation}
The dual line bundle $L^\ast_{\sT^\ast L^\times}\simeq \sJ^1L^\ast\times_M L^\ast$ can be identified with $\sT^\ast M\times \R^\ast\times\R^\ast$, whose points we denote $(p,z,Z)$. Sections of this bundle are now functions on $\sT^\ast M\times \R^\ast$, i.e., $Z=h(p,z)$. As usual in the trivial case, the first jet of a function $f$ is composed of the differential of $f$ and the value of $f$. We have therefore
$$\sJ^1L^\ast_{\sT^\ast L^\times}\simeq \sT^\ast\sT^\ast M\times \R^\ast \times \R^\ast\times \R\ni(\varphi, z, Z, Z_z).$$
If $\mh$ is a function on $\sT^\ast M\times \R^\ast$, then
$$\sj^1 \mathcal{h}=\left(\xd_{\sT^*M}\mathcal{h},z,\mathcal{h},\frac{\partial \mathcal{h}}{\partial z}\right),$$
where ${\xd}_{\sT^*M}$ is the differential in the direction of $\sT^*M$.
The manifold $\sJ^1L^\ast_{\sT^\ast L^\times}$ is the base for the principal bundle $\sT^\ast\sT^\ast L^\times$, with $\R^\times$ acting by the phase lift of ${\dts}k$. In the case $L^\times=M\times\R^\times$, we have the identification
$$\sT^\ast\sT^\ast L^\times\simeq\sT^\ast\sT^\ast M\times \R^\times\times\R^\ast\times\R^\ast\times\R\ni(\varphi,t,z,a_t,a_z).$$
The canonical map $\beta_{L^\times}$ for the trivial $L^\times$ reads
\begin{align*}
&\beta_{L^\times}:\sT\sT^\ast L^\times\simeq \sT\sT^\ast M\times \R^\times \times \R^\ast\times \R\times\R^\ast\\
&\longrightarrow \sT^\ast\sT^\ast M\times \R^\times \times \R^\ast\times \R^\ast\times\R\simeq \sT^\ast\sT^\ast L^\times,\\
&(w,t,z,\dot t, \dot z)\longmapsto (\beta_M(w),t,z,-\dot z, \dot t), \ \text{i.e.,}\;\; a_t=-\dot z, \; a_z=\dot t\,.
\end{align*}
The last map we need is the projection from $\sT^\ast\sT^\ast L^\times$ to $\sJ^1L^\ast_{\sT^\ast L^\times}$,
\begin{align*}
&\sT^\ast\sT^\ast L^\times \simeq  \sT^\ast\sT^\ast M\times \R^\times \times \R^\ast\times \R^\ast\times\R\\
& \longrightarrow
 \sT^\ast\sT^\ast M\times \R^\ast \times \R^\ast\times \R\simeq\sJ^1L^\ast_{\sT^\ast L^\times},\\
&(\varphi,t,z,a_t, a_z)\longmapsto
\left(\frac{1}{t}\odot\varphi,\; z,\; \frac{1}{t}\langle p,\, v\rangle+a_t,\; \frac{a_z}{t}\;\right),
\end{align*}
where $p$ and $v$ are the projections of $\varphi$ on $\sT^\ast M$ and $\sT M$, respectively. It means that
$$Z=\frac{1}{t}\langle p,v\rangle+a_t\ \text{and}\ Z_z=\frac{a_z}{t}.$$
Using the identifications described above, we can write
$$
\xymatrix@C+30pt{
(w,t,z,\dot t,\dot z)\ar@{|->}[r]^{\beta_{L^\times}}\ar@{|->}[d] & (\beta_M(w),t,z,-\dot z,\dot t)\ar@{|->}[d] \\
\left(\frac{1}{t}\odot w, z,\frac{\dot t}{t}, \dot z\right)\ar@{|->}[r]^(0.4){\beta_0} & \left(\frac{1}{t}\odot \beta_M(w), z, \frac{1}{t}\langle p,\, v\rangle-\dot z, \frac{\dot t}{t}\right).
}
$$
From the above diagram we read $\beta^{-1}_0$
as
\begin{align*}
\sT^\ast\sT^\ast M\times \R^\ast \times \R^\ast\times \R\ni(\varphi, z, Z, Z_z)&\longmapsto (\beta_M^{-1}(\varphi),z, Z_z,  \langle p,\, v\rangle-Z)\\
&\in
\sT\sT^\ast M\times \R^\ast \times \R\times \R^\ast\,.
\end{align*}
Composing $\beta_0^{-1}$ with the anchor (\ref{e:12}), we get the map $\beta$,
\bea&\beta:\sJ^1L_{\sT^\ast L^\times}\simeq \sT^\ast\sT^\ast M\times \R^\times \times \R^\times\times \R\notag\\
&\qquad\longrightarrow \sT\sT^\ast M\times\R^\ast\times\R^\ast\simeq\sT\sJ^1L^\ast, \notag \\
&(\varphi, z, Z, Z_z)\longmapsto (\beta_M^{-1}(\varphi)\oplus (-Z_zp), z,\langle p,v\rangle-Z).\label{e:14}
\eea
\end{example}

Applying $\beta$ to $\xd\mathcal{h}(x^i,p_j,z)$, i.e., putting $Z=\mathcal{h}$, $Z_k=\frac{\partial \mathcal{h}}{\partial x^k}$, $Z^l=\frac{\partial \mathcal{h}}{\partial p_l}$, and $Z_z=\frac{\partial \mathcal{h}}{\partial z}$,
we get
\be\label{hd}(\beta\circ\sj^1\mathcal{h})(x^i,p_j,z)=\left(x^i,\, p_j,\,z,\, \frac{\partial \mathcal{h}}{\partial p_k},\, -\frac{\partial \mathcal{h}}{\partial x^l}-\frac{\partial \mathcal{h}}{\partial z}p_l,\, p_n\frac{\partial \mathcal{h}}{\partial p_n}-\mathcal{h}\right),\ee
which, written as a vector field, gives exactly (\ref{e:2}).

\section{The Lagrangian side of the contact Tulczyjew triples}\label{sec:lag}
Completely analogously to the previous section, we can obtain the following proposition, this time for the Lagrangian part of the contact Tulczyjew triple.
\begin{theorem}
The Tulczyjew double vector bundle isomorphism
$$
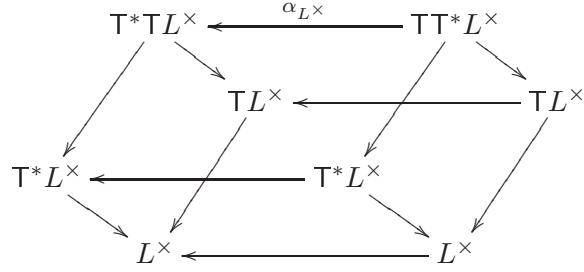
\xymatrix@C-20pt@R-10pt{
& \sT^\ast\sT \Lt  \ar[dr] \ar[ddl]
 & & & \sT\sT^\ast \Lt\ar[dr]\ar[ddl]\ar[lll]_{\za_\Lt}&  \\
 & & \sT \Lt\ar[ddl]
 & & & \sT \Lt \ar[ddl]\ar[lll]\\
 \sT^\ast \Lt\ar[dr]
 & & & \sT^\ast \Lt\ar[dr]\ar[lll] & &  \\
 & \Lt& & & \Lt\ar[lll] &
}$$
induces canonically the double vector bundle isomorphism described by the following commutative diagram of surjective vector bundle morphisms.
$$
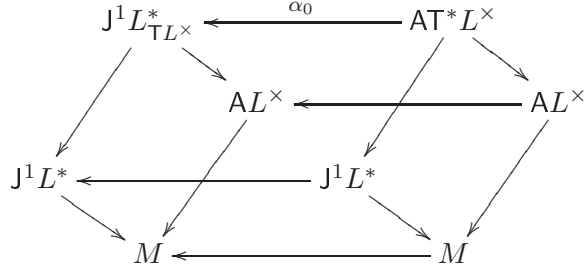
\xymatrix@C-18pt@R-10pt{
& \sJ^1 L^\ast_{\sT L^\times}  \ar[dr] \ar[ddl]
 & & & \sA\sT^\ast L^\times\ar[dr]\ar[ddl]\ar[lll]_{\za_0}&  \\
 & & \sA \Lt\ar[ddl]
 & & & \sA \Lt \ar[ddl]\ar[lll]\\
 \sJ^1L^*\ar[dr]
 & & & \sJ^1L^*\ar[dr]\ar[lll] & &  \\
 & M& & & M\ar[lll] &
}$$
The map $\za_0$ is additionally a contactomorphism with respect to the canonical contact structures
on $\sJ^1 L^\ast_{\sT L^\times}$ and $\sA\sT^\ast L^\times$.
Moreover, the composition of maps $\za=\cA\circ\za_0^{-1}$, where $\cA$ is the anchor of the Lie algebroid $\sA\sT^\ast L^\times$ (see~(\ref{cA}))  defines a morphism of double vector bundles
$$\za:\sJ^1 L_{\sT L^\times}\rightarrow \sT\sJ^1 {L^\ast},$$
with the corresponding commutative diagram of vector bundle morphisms
\be\label{Ls}\xymatrix@C-18pt@R-10pt{
& \sJ^1 L^\ast_{\sT L^\times} \ar[rrr]^\za \ar[dr] \ar[ddl]
 & & & \sT\sJ^1L^*\ar[dr]\ar[ddl]&  \\
 & & \sA \Lt\ar[rrr]\ar[ddl]
 & & & \sT M \ar[ddl]\\
 \sJ^1L^*\ar[dr]\ar[rrr]
 & & & \sJ^1L^*\ar[dr] & &  \\
 & M\ar[rrr]& & & M &
}\ee

\end{theorem}

\medskip\noindent The last diagram represents exactly the Lagrangian part of the (dynamical) contact Tulczyjew triple. Note that, in contrary to $\za_0$, the map
$\za$ is no longer a diffeomorphism and acts, in a sense, in the direction opposite to that of $\za_0$; however, it serves the same purposes as $\za_M$ in the case of Lagrangian mechanics.
To show how these maps look in coordinates, let us use the following ones:
\begin{itemize}
\item adapted local coordinates $(x^i,\tau, \dot x^j,\dot \tau)$ in $\sT L^\times$;
\item adapted local coordinates $(x^i,\tau, \dot x^j,\dot \tau, V_k, V_\tau, \dot V^j, \dot V_\tau)$ in $\sT^\ast\sT L^\times$;
\item local coordinates $(x^i,\dot x^j, t,\mu)$ in $L^\ast_{\sT L^\times}$;
\item adapted local coordinates $(x^i,\dot x^j, t,\mu_k, \dot\mu_l,\mu_t,\mu)$ in $\sJ^1L^\ast_{\sT L^\times}$;
\end{itemize}
if we use on $\sT\Lt$ the coordinates $(x^i,\zt,\dot x^j,t=\dot\zt/\zt)$ instead of $(x^i,\zt,\dot x^j,\dot\zt)$, then the adapted coordinates on $\sT^\ast\sT L^\times$ are
$$\left(x^i,\tau, \dot x^j,t, V_k, V_\tau-t\dot V_\tau, \dot V^j, t\dot V_\tau\right).$$ The projection $pr$ from  $\sT^*\sT\Lt$ onto $\sJ^1L^*_{\sT\Lt}$ is given by
$$
\left(x^i,\dot x^j,t,\zm_k,\dot\zm_l,\zm_t,\zm\right)\circ pr=\left(x^i,\dot x^j,t,\frac{V_k}{\zt},\frac{\dot V_l}{\zt},z,V_\tau-t\dot V_\tau\right).
$$
Starting form $\alpha_{L^\times}$ which reads
$$
\left(x^i,\tau, \dot x^j,\dot \tau, V_k, V_\tau, \dot V^l, \dot V_\tau\right)\circ\alpha_{L^\times}=\left(x^i, \tau,\dot x^j,\dot \tau, \dot \pi_j, \dot z, \pi_l, z\right),
$$
we get the isomorphism $\za_0:\sA\sT^\ast L^\times\to\sJ^1 L^\ast_{\sT L^\times}$ in the form
$$
\left(x^i,\dot x^j,t,\zm_k,\dot\zm_l,\zm_t,\zm\right)\circ\za_0=
\left(x^i,\dot x^j,t,\dot p_k+tp_k,p_l,z,\dot z-tz\right).
$$
Consequently, the map $\za:\sJ^1 L^\ast_{\sT L^\times}\to\sT\sJ^1L^*$ reads
$$
(\,x^i,\,p_j,\,z,\,\dot x^k,\,\dot p_l,\,\dot z\,)\circ\za=(x^i,\, \dot\zm_j,\,\zm_t,\, \dot x^k,\, \zm_l-t\dot\zm_l,\, \zm+t\zm_t).
$$
The double vector bundle structure on $\sJ^1 L^\ast_{\sT L^\times}$ can be presented as follows.
$$
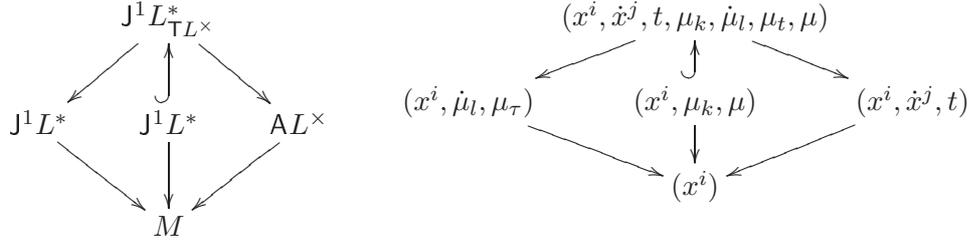
\xymatrix@C-10pt{
 & \sJ^1L^\ast_{\sT L^\times}\ar[dl]\ar[dr] & \\
\sJ^1L^\ast\ar[dr] &  \sJ^1L^\ast\ar@{ (->}[u]\ar[d] & \sA L^\times\ar[dl] \\
 & M &
}\qquad
\xymatrix@C-20pt@R-10pt{
 & (x^i,\dot x^j,t,\zm_k,\dot\zm_l,\zm_t,\zm)\ar[dl]\ar[dr] & \\
(x^i,\dot\zm_l,\zm_\tau)\ar[dr] & (x^i,\zm_k,\zm)\ar@{ (->}[u]\ar[d] & (x^i,\dot x^j,t)\ar[dl] \\
 & (x^i) &
}
$$
Consequently, the isomorphism $R:\sJ^1L^*_{\sT\Lt}\to\sJ^1L^*_{\sT^*\Lt}$, an analog of $R_{\sT\Lt}$~(\ref{ci}), takes the form
$$
\left(x^i,p_j,z,Z_k,Z^l,Z_z,Z\right)\circ R=\left(x^i,\dot\zm_j,\zm_t,-\zm_k,\dot x^l,t,
\dot\zm_r\dot x^r+t\zm_t-\zm\right).
$$

\begin{example} Following Example \ref{ex:1}, we construct the map $\alpha$ for the trivial line bundle $L=M\times\R$. For that, we have to examine all the spaces and maps from the diagram (\ref{Ls}), assuming the triviality of $L$.
The Atiyah algebroid $\sA(M\times\R^\times)$ can be identified with $\sT M\times\R$ with points $(v,\tau)$, and the projection $\sT L^\times\rightarrow \sA L^\times$ reads
$$
\sT M\times \R^\times\times\R\ni(v,\zt,\dot\zt)\longmapsto \left(v,t={\dot\zt}/{\zt}\right)\in\sT M\times \R.
$$
The above projection is, in fact, the projection in the principal bundle with the $\R^\times$-action {${\dt} k$}. 1-homogeneous functions on $\sA L^\times$ are identified with sections of the dual line bundle
$$L^\ast_{\sT L^\times}\simeq \sT M\times\R\times\R^\ast$$
with elements denoted by $(v,t,\mu)$. The first jet bundle will then be
$$\sJ^1 L^\ast_{\sT L^\times}\simeq\sT^\ast\sT M\times\R\times\R^\ast\times\R^\ast\ni(\psi,t,\mu,\mu_t).$$
Next, we need the projection from $\sT^\ast \sT L^\times$ to $\sJ^1 L^\ast_{\sT L^\times}$. For that, we identify $\sT^\ast \sT L^\times$ with $\sT^\ast\sT M\times \R^\times \times\R\times \R^\ast\times\R^\ast$. The projection reads
\begin{align*}
&\sT^\ast\sT M\times \R^\times \times\R\times \R^\ast\times\R^\ast\ni (\varphi,\,\zt,\,\dot\zt,\, c_\zt,\, c_{\dot\zt})\\
&\qquad\longmapsto \left(\frac{\varphi}{\zt},\, t,\, c_\zt+tc_{\dot\zt},\, c_{\dot \zt}\right)\in \sT^\ast\sT M\times \R\times \R^\ast\times\R^\ast,
\end{align*}
which means in particular that $\mu=c_\zt+tc_{\dot\zt}$ and $\mu_t=c_{\dot\zt}$.
Finally, we are ready to write the map $\alpha:\sJ^1L^\ast_{\sT L^\times}\rightarrow \sT\sJ^1 L^\ast$ for the trivial bundle ${L\!=M\times\R}$:
\beas
\alpha:\sJ^1L^\ast_{\sT L^\times}&\simeq& \sT^\ast\sT M\times \R\times \R^\ast\times\R^\ast\longrightarrow  \sT\sT^\ast M\times \R^\ast\times \R^\ast\simeq\sT\sJ^1 L^\ast \\
&&(\psi,\,t,\,\mu,\,\mu_t)\longmapsto \left(\alpha_M^{-1}(\psi)\oplus(-t p),\, \mu_t,\, \mu-t\mu_t\right).
\eeas
In the above formula $p=\tau_{\sT^\ast M}(\alpha_M^{-1}(\varphi))\in\sT^\ast M$.

\medskip\noindent In the contact Lagrangian mechanics, a Lagrangian is a section of the line bundle $L^\ast_{\sT L^\times}$. In our case this bundle is trivial,
$$L^\ast_{\sT L^\ast}\simeq\sT M\times\R\times\R^\ast\rightarrow\sT M\times\R.$$
Every section of this bundle can then be identified with a function
$$\sT M\times\R\ni(v,t) \mapsto\ell(v,t)\in\R^\ast.$$
The first jet of the section is composed of the value of $\ell$ and its differential. In the identification
$$\sJ^1 L^\ast_{\sT L^\times}\simeq\sT^\ast\sT M\times \R\times \R^\ast\times\R^\ast,$$
the first jet reads
$$\sj^1\ell(v,t)=\left({\xd}_{\sT M}\ell(v,t),\, t, \ell(v,t), \frac{\partial \ell}{\partial t}\right),$$
where $\xd_{\sT M}\ell$ is the differential of $\ell$ in the direction of $\sT M$. If we use the adapted coordinates $(x^i,p_j,\dot x^k,\dot p_l)$ in $\sT\sT^\ast M$, then the image of $\sj^1\ell$ by $\alpha$ in the identification
$$\sT\sJ^1L^\ast\simeq \sT\sT^\ast M\times\R^\ast\times\R^\ast$$
is determined by the equations
\be\label{ld}
p_i=\frac{\partial\ell}{\partial \dot x^i},\qquad
\dot{p_j}=\frac{\partial\ell}{\partial x^j}-t\frac{\partial\ell}{\partial \dot x^j},\qquad
z=\frac{\partial\ell}{\partial t},\qquad
\dot z=\ell-t\frac{\partial\ell}{\partial t}.
\ee
From these equations, we can deduce the second-order differential equations for a curve in $M\times \R$ in the form
$$
\frac{d}{ds}\left(\frac{\partial\ell}{\partial\dot x}\right)=\frac{\partial\ell}{\partial x^j}-t\frac{\partial\ell}{\partial \dot x^j},\qquad
\frac{d}{ds}\left(\frac{\partial\ell}{\partial t}\right)=\ell-t\frac{\partial\ell}{\partial t},
$$
with $s$ being the parameter of the curve.
They can be viewed as \emph{contact Euler-Lagrange equations}.
Note that with respect to $t$ the above equations are of the first order.
\end{example}

\section{The contact Legendre map}
In our approach, elements of $\sJ^1L^\ast_{\sT L^\times}$ are first jets of Hamiltonians, which are maps
$$
\xymatrix{
\sJ^1L^\ast\ar[r]^{\mh}\ar[d] & L^\ast\ar[d] \\
M\ar[r]^{\id_M} & M
}$$
over the identity on $M$. In the adapted coordinates, $\mh$ is represented by a function  $\mh=\mh(x^i,p_j,z)$ on $L^\ast_{\sT^\ast L^\times}$. The first jet of $\mh$ is just the tangent map at a fixed point, i.e., $\sj^1\mh(u)$, for a fixed $u=\sj^1\sigma(x)\in \sJ^1L^\ast$, is a map
$$
\xymatrix{
\sT_{u}\sJ^1L^\ast\ar[r]\ar[d] & \sT_{\mh(u)} L^\ast\ar[d] \\
\sT_xM\ar[r]^{\id} & \sT_xM
}$$
over the identity on $\sT_xM$. Restricting it to vectors vertical with respect to the projection onto $M$, we get a map $\sV_{u}\sJ^1L^\ast\rightarrow\sV_{\mh(u)} L^\ast$. The both bundles
$\sJ^1L^\ast\rightarrow M$ and $L^\ast\longrightarrow M$ are vector bundles, therefore we can use the fact that vertical vectors can be identified with elements of the fiber. We get a map
$$\sJ^1L^\ast\times_M\sJ^1L^\ast\rightarrow L^\ast\times_M L^\ast, \qquad (u, \delta u)\longmapsto (\mh(u), \xd^\sv \mh(u,\delta u)).$$
For fixed $u$, the second part of the above map, namely $\sJ^1 L^\ast\ni\delta u\rightarrow \xd^\sv \mh(u,\delta u)\\\in L^\ast$, defines an element $\xd^\sv \mh(u)\in\sJ^1L\ot L^*\simeq\sA L^\times$,
according to (\ref{iso}). This gives a contact version of the Legendre map on the Hamiltonian side,
$$\zl_\mh:\sJ^1L^*\to\sA\Lt,\quad \zP(\zl_\mh(u),\zd u)=\xd^\sv \mh(u,\zd u),$$
where $\zP$ is the canonical pairing (\ref{pairing}).
This is a smooth map covering the identity on $M$. We call the Hamiltonian $\mh$ \emph{hyperregular} if $\zl_\mh$ is actually a diffeomorphism.

\medskip\noindent Since $L^\ast_{\sT L^\times}\simeq \sA L^\times\times_M L^\ast$, Lagrangians are sections of $L^\ast_{\sT L^\times}$, which are maps
$$
\xymatrix{
\sA L^\times\ar[r]^\ell\ar[d] & L^\ast\ar[d] \\
M\ar[r]^{\id_M} & M
}$$
covering the identity on $M$. We fix $v\in\sA L^\times$ and consider $\sj^1\ell(v)$ as a map
$$
\xymatrix{
\sT_{v}\sA L^\times \ar[r]\ar[d] & \sT_{\ell(v)} L^\ast\ar[d] \\
\sT_xM\ar[r]^{\id} & \sT_xM
}$$
over the identity on $\sT_xM$. Restricting the above map to vertical vectors and using again the description of the vertical tangent bundle for a vector bundle, we get a map
$$\sA L^\times\times_M\sA L^\times\longrightarrow L^\ast\times_M L^\ast, \qquad (v, \delta v)\longmapsto (\ell(v), \xd^\sv \ell(v,\delta v)).$$
For a fixed $v$, the second part of the above map determines an element of $\sA^*\Lt\ot L^*\simeq\sJ^1L^\ast$ (cf. (\ref{iso})). This defines the contact version of the Legendre map:
$$\lambda_\ell:\sA L^\times\longrightarrow \sJ^1 L^\ast,\quad \zP(\zd v,\lambda_\ell(v))=\xd^\sv \ell(v,\delta v). $$
This is again a smooth map covering the identity on $M$. We will call $\ell$ \emph{hyperregular} if $\zl_\ell$ is actually a diffeomorphism.
\begin{example}
In local coordinates $(x^i,\dot x^j,t)$ on $\sA\Lt$ and $(x^i,p_j,z)$ on $\sJ^1L^*$, a Lagrangian
$\ell$ is represented by a function $\ell(x^i,\dot x^j,t)$. The pairing $\zP$ (\ref{pairing1}) with values in $L^*$ is represented by the function
$$\zP(x^i,\dot x^j,t,p_k,z)=\dot x^jp_j+tz.$$
Since
$$\xd^\sv\ell\left((x^i,\dot x^j,t),(x^j,(\dot x^j)',t')\right)=\frac{\pa\ell}{\pa \dot x^j}(x^i,\dot x^j,t)\cdot (\dot x^j)'
+\frac{\pa\ell}{\pa t}(x^i,\dot x^j,t)\cdot t',$$
the Legendre map $\lambda_\ell:\sA L^\times\longrightarrow \sJ^1 L^\ast$ reads
$$\zl_\ell(x^i,\dot x^j,t)=\left(x^i,\frac{\pa\ell}{\pa \dot x^j}(x^i,\dot x^j,t),\frac{\pa\ell}{\pa t}(x^i,\dot x^j,t)\right).$$
Similarly, we obtain that for a Hamiltonian $\mh$ the Legendre map $\zl_\mh:\sJ^1L^*\!\!\to\sA\Lt$ reads
$$\zl_\mh(x^i,p_j,z)=\left(x^i,\frac{\pa\mh}{\pa{p_j}}(x^i,p_j,z),\frac{\pa\mh}{\pa z}(x^i, p_j,z)\right).$$
\end{example}
\begin{proposition}\label{pr1}
If $\ell:\sJ^1L^*\to L^*$ is a hyperregular Lagrangian, then
$$\mh(u)=\zP\left(u,\zl^{-1}_\ell(u)\right)-(\ell\circ\zl^{-1}_\ell)(u)$$
is a hyperregular Hamiltonian and $\zl_\mh=\zl^{-1}_\ell$.
\end{proposition}
\begin{proof}
Since in local coordinates $\zl_\ell$ is just the vertical derivative and $\zP:\sA\Lt \ti_M\sJ^1L^*\to L^*$ looks like the standard scalar product (see (\ref{Pi})), the proof
in coordinates is formally the same as in the standard mechanics, so we skip the calculations.
\end{proof}
\begin{corollary}\label{c1}
If $\mh:\sA\Lt\to L^*$ is a hyperregular Hamiltonian, then
$$\ell(v)=\zP\left(\zl^{-1}_\mh(v),v\right)+(\mh\circ\zl^{-1}_\mh)(u)$$
is a hyperregular Lagrangian.
\end{corollary}
\enlargethispage{3em}
\begin{example}
Let us suppose that the line bundle $L$ is trivializable so that we can represent $L$ as $M\ti\R$. Consequently, we can use identifications $L^*=M\ti\R^*$, $\sJ^1L^*=\sT^*M\ti\R^*$, and
$\sA\Lt=\sT M\ti\R$. The pairing $\zP$ can be identified with the canonical pairing (real-valued) between $\sJ^1L^*=\sT^*M\ti\R^*$, $\sA\Lt=\sT M\ti\R$, and Lagrangians $\ell$ with functions on $\sT M\ti\R$. Hence, for vectors $v_x,v'_x\in\sT_xM$, we have
$$\xd^\sv\ell\left((v_x,t)(v'_x,t')\right)=\la\xd^\sv_{\sT M}\ell(v_x,t),v'_x\ran+
\frac{\pa\ell}{\pa t}(v_x,t)\cdot t',$$
where $\xd^\sv\ell(v_x,t)$ is the vertical derivative of the function $\ell(\cdot,t)$ on $\sT M$.
The corresponding Legendre map is
$$\zl_\ell(v_x,t)=\left(\xd^\sv_{\sT M}\ell(v_x,t),\frac{\pa\ell}{\pa t}(v_x,t)\right).$$
Let us choose a Riemannian metric $g$ on $M$, and consider the Lagrangian
$$\ell(v_x,t)=\frac{1}{2}\left(\Vert v_x\Vert^2+t^2\right),$$
where $\Vert\cdot\Vert$ is the norm corresponding to $g$. Then
$$\zl_\ell(v_x,t)=\left(g(v_x,\cdot),t\right),$$
which is clearly an isomorphism of the vector bundle $\sT M\ti\R$ onto $\sT^*M\ti\R^*$, so $\ell$ is hyperregular.
\end{example}
\begin{example}\label{regL}
We will construct an example of a hyperregular Lagrangian $\ell:\sJ^1B\to B^*$ for the M\"obius band $B\to S^1$ which is a non-trivializable line bundle.
\enlargethispage{3em}
Similarly as in Example \ref{AB}, we will use two charts $\cO_1$ and $\cO_2$  on $B$ with coordinates $(x,\zt)\in]0,\pi[\ti\R$ and $(x',\zt')\in]\frac{\zp}{2},\frac{3\zp}{2}[\ti\R$, respectively.

The intersection $\cO_{12}=\cO_1\cap\cO_2$ consists of two disjoint sets $\cO_{12}=\cO_{12}^1\cup\cO_{12}^2$ represented by $]\frac{\zp}{2},\pi[\ti\R$ in $\cO_1$ and $\cO_2$, and $\cO_{12}^2$ represented by $]0,\frac{\zp}{2}[\ti\R$ in $\cO_1$ and with $]\pi,\frac{3\zp}{2}[$ in $\cO_2$, with the transition functions
$(x',\zt')=(x,\zt)$ on $\cO_{12}^1$, and $(x',\zt')=(x+\pi,-\zt)$ on $\cO_{12}^2$.
The dual coordinates $(x,z)$ in the dual line bundle $B^*$ transform formally in the same way, if we replace $\zt$ and $\zt'$ with $z$ and $z'$, respectively.

The adapted coordinates $(x,\dot x,y=\dot\zt/\zt)$ and $(x',\dot x',y')$ in $\sA B^\ti$ are defined on the corresponding open subsets $\cU_1$ and $\cU_2$ of $\sA B^\ti$, having the intersection $\cU_{12}$ composed with two disjoint parts, $\cU_{12}=\cU_{12}^1\cup\cU_{12}^2$. The chart $\cU_{12}^1$ is represented by $]\frac{\zp}{2},\pi[\ti\R\ti\R$ in $\cU_1$ and $\cU_2$, and $\cU_{12}^2$ represented by $]0,\frac{\zp}{2}[\ti\R\ti\R$ in $\cU_1$ and with $]\pi,\frac{3\zp}{2}[\ti\R\ti\R$ in $\cU_2$, with the trivial transition function on $\cU_{12}^1$, and the transition $(x',\dot x',t')=(x+\pi,\dot x,t)$ on $\cU_{12}^2$.

Completely analogously, the transition functions for the corresponding charts $\cV_1$ and $\cV_2$ in $\sJ^1B^*$ are: the identity $(x',p',z')=(x,p,z)$ on $\cV_{12}^1$
and
\begin{align*}
]0,3\pi/2[\ti\R\ti\R&\ni(x,p,z)\mapsto(x'=x+\pi,p'=-p,z'=-z)\\
&\in ]\pi,{3\zp}/{2}[\ti\R\ti\R
\end{align*}
on $\cV_{12}^2$.
A Lagrangian $\ell:\sA B^\ti\to B^*$ is represented by a function $z\!=\!F_1(x,\dot x,t)$ on $\cU_1$
and a function $z'=F_2(x',\dot x',t)$ on $\cU_2$. These two functions define consistently a Lagrangian if and only if $F_1=F_2$ on $]\frac{\zp}{2},\pi[\ti\R\ti\R$ and
\be\label{F}F_2(x+\pi,\dot x,t)=-F_1(x,\dot x,t)\ee
for $x\in]0,\frac{\zp}{2}[$.
Then the Legendre map $\zl_\ell$ is consistently represented by a map $\cU_1\to\cV_1$,
$$(x,p,z)=\left(x,\frac{\pa F_1}{\pa\dot x},\frac{\pa F_1}{\pa t}\right),$$
and a map $\cU_2\to\cV_2$,
$$(x',p',z')=\left(x',\frac{\pa F_2}{\pa\dot x'},\frac{\pa F_2}{\pa t'}\right).$$
Let us consider
$$F_1(x,\dot x,t)=\frac{\cos(x)}{2}\left(\dot x^2-t^2\right)+\sin(x)t\dot x$$
and $F_2$ which is formally the same, but with a different domain,
$$F_2(x',\dot x',t')=\frac{\cos(x')}{2}\left((\dot x')^2-(t')^2\right)+\sin(x')t'\dot x'.$$
These functions clearly coincide for $x{=x'}\in]\frac{\zp}{2},\zp[$. For $x\in]0,\frac{\zp}{2}[$ we have
$$F_2(x+\pi,\dot x,t)=\frac{\cos(x+\pi)}{2}\left(\dot x^2-t^2\right)+\sin(x+\pi)t\dot x=-F_1(x,\dot x,t).$$
The compatibility condition is satisfied, so these two functions define uniquely a Lagrangian $\ell$ for $B$. The hyperregularity of $\ell$ is now equivalent to the hyperregularity of these two functions. They are given formally by the same formula, so we need to check the hyperregularity of the function $F$ defined formally as $F_1$ but on the domain $\R^3$. We have
\begin{align}
\label{hr}\zl_F(x,\dot x,t)&=\left(x,\frac{\pa F}{\pa\dot x},\frac{\pa F}{\pa t}\right)\nonumber\\
&=\big(x,\cos(x)\dot x+\sin(x)t,\sin(x)\dot x-\cos(x)t\big).
\end{align}
The matrix
$$\begin{bmatrix}
\cos(x) & \ \sin(x)\\
\sin(x) & -\cos(x)\end{bmatrix}$$
is orthogonal for all $x\in\R$, so map (\ref{hr}) is a diffeomorphism.
\end{example}
\begin{remark}
The problem of the existence of hyperregular Lagrangians/Ha\\miltonians for an arbitrary $L$ seems to be difficult and substantially extends the purposes of this paper, so we postpone it to a further development of the theory.
\end{remark}

\section{The contact Tulczyjew triples}\label{sec:trip}
Starting with a line bundle $L\to M$, and combining the Hamiltonian and the Lagrangian parts described in the last two sections, we obtain a contact Tulczyjew triple. The geometric part, which comes from the reduction of the classical triple for the manifold $\Lt$ by the $\Rt$-actions, consists, in the top line, of contactomorphism which are simultaneously isomorphisms of double vector bundles:
\be\xymatrix@C-20pt@R-5pt{
 &\sJ^1L^*_{\sT^*\Lt}
\ar[ddl] \ar[dr]
 &  &  & \sA \sT^\ast \Lt \ar[lll]_{\zb_0}\ar[rrr]^{\za_0}\ar[ddl] \ar[dr]
 &  &  & \sJ^1L^*_{\sT\Lt} \ar[ddl] \ar[dr]
 & \\
 & & \sA\Lt \ar[ddl]
 & & & \sA\Lt\ar[rrr]\ar[lll]\ar[ddl]
 & & & \sA\Lt\ar[ddl]
 \\
 \sJ^1L^* \ar[dr]
 & & & \sJ^1L^*\ar[dr]\ar[rrr]\ar[lll]
 & & & \sJ^1L^*\ar[dr]
 & & \\
 & M
 & & & M \ar[rrr]\ar[lll]& & & M &
}\label{gTt}\ee

\medskip\noindent The composition
$$\zb_0\circ\za_0^{-1}:\sJ^1L^*_{\sT\Lt}\;\to\;\sJ^1L^*_{\sT^*\Lt}$$
is a double vector bundle isomorphism which is the contact analog of
the Tulczyjew isomorphism (\ref{ci}).
Much more important for applications to mechanics is the derived dynamical part:
\be\xymatrix@C-20pt@R-5pt{
 &\sJ^1L^*_{\sT^*\Lt}\ar[rrr]^{\zb}
\ar[ddl] \ar[dr]
 &  &  & \sT\sJ^1L^* \ar[ddl] \ar[dr]
 &  &  & \sJ^1L^*_{\sT\Lt} \ar[lll]_{\za}\ar[ddl] \ar[dr]
 & \\
 & & \sA\Lt \ar[ddl]\ar[rrr]
 & & & \sT M\ar[ddl]
 & & & \sA\Lt\ar[ddl]\ar[lll]
 \\
 \sJ^1L^* \ar[dr]\ar[rrr]
 & & & \sJ^1L^*\ar[dr]
 & & & \sJ^1L^*\ar[dr]\ar[lll]
 & & \\
 & M\ar[rrr]
 & & & M & & & M \ar[lll]&
}\label{dTt}\ee
The middle double vector bundle, $\sT\sJ^1L^*$, is the tangent bundle of the contact phase space $\sJ^1L^*$ whose subsets are interpreted as implicit dynamics on the contact phase space, but this time the maps $\za,\zb$ are no longer diffeomorphisms. Note that the contact phase space carries a canonical contact structure, but $\sT\sJ^1L^*$ is not a contact manifold, since it is even-dimensional. The left-hand side of the triple corresponds to the Hamiltonian formalism, while the right-hand side corresponds to the Lagrangian formalism. Let us recall the canonical isomorphisms (\ref{e:21})
$$L^\ast_{\sT^\ast L^\times}\simeq \sJ^1L^\ast\times_M L^\ast\quad\text{and}\quad  L^\ast_{\sT L^\times}\simeq \sJ^1L^\ast\times_M L^\ast.$$

\medskip\noindent
In this picture, Hamiltonians are sections of the line bundle
$$L^\ast_{\sT^\ast L^\times}\simeq\sJ^1L^\ast\times_M L^\ast\longrightarrow \sJ^1L^\ast,$$
or, equivalently, maps $\mh: \sJ^1L^\ast\longrightarrow L^\ast$ over the identity on $M$. For a given Hamiltonian $\mh$ the dynamics is the image of the contact Hamiltonian vector field, and can be written as $$\mathcal{D}_{\mh}=\beta\left(\sj^1 {\mh}(\sJ^1L^\ast)\right)\subset\sT\sJ^1L^\ast.$$
Lagrangians are sections of the line bundle
$$L^\ast_{\sT L^\times}\simeq\sA L^\times\times_ML^\ast\rightarrow \sA L^\times,$$
or, equivalently, maps $\ell: \sA L^\times\rightarrow L^\ast$ over the identity on $M$. The Lagrangian dynamics is obtained directly from a Lagrangian by means of the map $\alpha$, namely $$\mathcal{D}_\ell=\alpha\left(\sj^1 \ell(\sA L^\times)\right).$$
Note that $\sj^1 \mh(\sJ^1L^\ast)$ and $\sj^1 \ell(\sA L^\times)$ are Legendre submanifolds
in the contact manifolds $\sJ^1L^*_{\sT^*\Lt}$ and $\sJ^1L^*_{\sT\Lt}$, respectively.

The purpose of the classical Tulczyjew triple was to deal with singular Lagrangians.  More precisely, to provide tools for obtaining the phase dynamic for Lagrangians for which there is no Hamiltonian, understood as a single function on the phase space. In such cases, dynamics is not the image of a Hamiltonian vector field but a more general Lagrangian submanifold. This submanifold serves as an implicit differential equation for phase trajectories. Here, we have exactly the same possibility to consider Lagrangians which are singular.

\subsection{The contact Legendre transformation}
In analytical mechanics, the Legendre transformation is a procedure of passing from a Lagrangian description of a mechanical system to a Hamiltonian description of this system. This process is associated with the symplectic structures of the Hamiltonian and Lagrangian sides of the classical Tulczyjew triple. In the classical case, with the configuration space being $\sT M$ and the phase space being $\sT^\ast M$, a Lagrangian function (or a more general generating object) generates a Lagrangian submanifold $\mathcal{D}_L$ of $\sT^\ast\sT M$. This submanifold can be mapped by $\alpha_M^{-1}$ to $\sT\sT^\ast M$, producing the dynamics $\mathcal{D}=\alpha_M^{-1}(\mathcal{D}_L)$. We can ask if there exists a Hamiltonian generating object $H$ for the same dynamics, more precisely for a Lagrangian submanifold $\mathcal{D}_H$ of $\sT^\ast\sT^\ast M$, such that $\mathcal{D}=\beta_M^{-1}(\mathcal{D}_H)$. If we allow for families of function as generating objects for Hamiltonian dynamics, then the answer is `yes'. There is also a universal formula for such a generating family: for a given Lagrangian $L:\sT M\rightarrow M$ we take
$$H:\sT M\times_M\sT^\ast M\longrightarrow \R, \quad H(v,p)=\langle p,v\rangle - L(v).$$
It is a family of functions on the phase space parameterized by elements of $\sT M$. In many cases, this family can be simplified. For hyperregular Lagrangians, it can be simplified to just one function on the phase space.
The Legendre transformation can also be performed from Hamiltonian to Lagrangian generating objects. For a given Hamiltonian function $H$, we can take the Lagrangian generating family of the form
$$L:\sT^\ast M\times_M\sT M\longrightarrow \R, \quad L(p,v)=\langle p,v\rangle - H(p),$$
and simplify it if possible.

The fundamental observation that a contact structure is a homogeneous symplectic structure on an $\Rt$-principal bundle allows us to perform the Legendre transformation for contact mechanical systems completely analogously. It is easy to see that, if we start from a 1-homogeneous Hamiltonian $H$ on $\sT^\ast L^\times$, we get a 1-homogeneous Lagrangian generating object. Indeed, if $H(({\dts}k)_s(p))=sH(p)$, then
\begin{align*}
L(({\dts}k)_s(p),({\dt}k)_s(v))&=\langle ({\dts}k)_s(p),({\dt}k)_s(v))\rangle - H(({\dts}k)_s(p))\\
&=\left\langle s\sT^\ast k_{s^{-1}}(p),\sT k_s(v)\right\rangle -sH(p)\\
&=
s\left(\langle p,v\rangle-H(p)\right)=sL(p,v).
\end{align*}
We can also work with sections of line bundles instead of homogeneous functions, and with Legendre submanifolds instead of Lagrangian homogeneous submanifolds.
The following theorem describes conditions under which a single Lagrangian and a single Hamiltonian determine the same dynamics. This is completely analogous to those in classical mechanics.
\begin{theorem}
Let $\mh\in\Sec(L^\ast_{\sT^\ast L^\times})$ be a Hamiltonian section and $\ell\in\Sec\\(L^\ast_{\sT L^\times})$ be a Lagrangian section. Then $\cD_\mh=\cD_\ell$ if and only if $\ell$ is hyperregular (equivalently, $\mh$ is hyperregular) and
\be\label{cDmh}\mh(u)=\zP\left(u,\zl^{-1}_\ell(u)\right)-(\ell\circ\zl^{-1}_\ell)(u);\ee
equivalently,
$$\ell(v)=\zP\left(\zl^{-1}_\mh(v),v\right)+(\mh\circ\zl^{-1}_\mh)(u),$$
where $\zl_\mh=\zl^{-1}_\ell$.
\end{theorem}
\begin{proof}
According to (\ref{hd}) and (\ref{ld}), the identification of dynamics  $\cD_\mh=\cD_\ell$ is equivalent to the system of PDEs, which in coordinates reads
\begin{align}\label{mhe}
&\left(x^i,\, p_j,\,z,\, \frac{\partial \mh}{\partial p_k},\, -\frac{\partial \mh}{\partial x^l}-\frac{\partial \mh}{\partial z}p_l,\, p_n\frac{\partial \mh}{\partial p_n}-\mh\right)\\
&\quad=\left(x^i,\frac{\partial\ell}{\partial \dot x^l},\frac{\partial\ell}{\partial t},\dot x^k,
\frac{\partial\ell}{\partial x^l}-t\frac{\partial\ell}{\partial \dot x^l},
\ell-t\frac{\partial\ell}{\partial t}\right).\nonumber
\end{align}
This immediately implies that
$$(x^i, p_j,z)=\zl_\ell(x^i,\dot x^j,t),
$$
so $\zl_\ell$ is a diffeomorphism and $\ell$ is hyperregular. The equation
$$p_n\frac{\partial \mh}{\partial p_n}-\mh=\ell-t\frac{\partial\ell}{\partial t}$$
implies
$$\mh(x,p,z)=p_n\dot x^n+tz-\ell(x,\dot x,t),$$
where we identify $(x^i, p_j,z)$ with $(x^i,\dot x^j,t)$ \emph{via} $\zl_\ell$. This is exactly (\ref{cDmh}). According to Proposition \ref{pr1}, $\mh$ is hyperregular and $\zl_\mh=\zl^{-1}_\ell$.
The rest follows directly from Corollary \ref{c1}.

Conversely, if we know that $\mh$ and $\ell$ are hyperregular and related by (\ref{cDmh}), then
the only not checked equation in (\ref{mhe}) is
$$-\frac{\partial \mh}{\partial x^l}-\frac{\partial \mh}{\partial z}p_l=\frac{\partial\ell}{\partial x^l}-t\frac{\partial\ell}{\partial \dot x^l},$$
which reduces to
$$-\frac{\partial \mh}{\partial x^l}(x,p,z)=\frac{\partial\ell}{\partial x^l}(x,\dot x,t).$$
The latter follows easily by direct calculations.
\end{proof}
\subsection{A comparison with other concepts of contact Tulczyjew triples}

We have constructed a contact Tulczyjew triple starting from the classical Tulczyjew triple for the manifold $P=L^\times$, and dividing the diagram (\ref{zb}) by the action of $\R^\times$. As a result, we have obtained the diagram (\ref{gTt}) and its dynamical consequence (\ref{dTt}). In the case when $L^\times$ is the trivial bundle $L^\times=M\times \R^\times$, one can propose an alternative version of the Lagrangian side of the triple which was first constructed in \cite{Esen:2021}. Let us consider the following principal $\R^\times$ bundle,
$$\sT M\times\sT^\ast\R^\times\simeq \sT M\times\R^\times\times\R^\ast\longrightarrow \sT M\times\R^\ast,\qquad
(v,t,z)\longmapsto (v,z)$$
with the $\Rt$-action
$$\chi_s (v,t,z)=(v,st,z).$$
The phase lift ${\dts}\chi$ acts on $\sT^\ast(\sT M\times\sT^\ast\R^\times)\simeq \sT^\ast\sT M\times\sT^\ast\sT^\ast\R^\times$ by
$$ ({\dts}\chi)_s(\varphi,t,z,f_t,f_z)=(s\varphi, st, z,  f_t, sf_z),$$
where $(t,z,f_t,f_z)$ is an element of
$$\sT^\ast\sT^\ast\R^\times\simeq \R^\times\times\R^\ast\times\R^\ast\times \R.$$
The canonical symplectic structure on
$\sT^\ast(\sT M\times\sT^\ast\R^\times)$ reads
$$\omega_{\sT M}+\xd f_t\wedge\xd t+\xd f_z\wedge \xd z$$
and is clearly 1-homogeneous with respect to ${\dts}\chi$. This means that there is a contact structure on the base of the principal $\R^\times$-bundle $\sT^\ast(\sT M\times\sT^\ast\R^\times)$ which, as previously, can be identified with the space of first jets of some line bundle. To identify this bundle, let us note that 1-homogeneous functions on $\sT M\times\sT^\ast\R^\times$ are of the form
$$(v,t,z)\longmapsto tF(v,z)$$
for some $F:\sT M\times\R^\ast\rightarrow \R$. Such functions correspond to sections of the trivial line bundle $\sT M\times\R^\ast\times\R\to \sT M\times\R^\ast$, which can be more conveniently expressed as functions on $\sT M\times\R^\ast$. The homogeneous function $(v,t,z)\mapsto tF(v,z)$  corresponds to the function $(v,z)\mapsto F(v,z)$.
Since $\sT M\times\R^\ast\times\R\to \sT M\times\R^\ast$ is trivial, the space of first jets of its sections is of the form
$$\sJ^1(\sT M\times\R^\ast\times\R)=\sT^\ast(\sT M\times\R^\ast)\times \R\simeq \sT^\ast\sT^\ast M\times \R^\ast\times \R\times\R,$$
where the first jet of $F$ is
$$\left(\xd_{\sT M}F(v),z, \frac{\partial F}{\partial z}, F\right).$$
This space serves as an alternative Lagrangian side of the contact Tulczyjew triple in the trivial case. Using coordinates $(x^i, \dot x^i, b_{x^j}, b_{\dot x^j}; z,u, w)$ in $\sT^\ast\sT^\ast M\times \R^\ast\times \R\times\R$, we have then
$$b_{x^i}=\frac{\partial F}{\partial x^i}, \quad b_{\dot x^j}=\frac{\partial F}{\partial x^j}, \quad u=\frac{\partial F}{\partial z}, \quad w=F(x^i,\dot x^j, z).$$

To finish the construction we shall need an alternative version of $\alpha$, which we will denote $\bar\alpha$. Note that $\sT^\ast(\sT M\times\sT^\ast\R^\times)\simeq \sT^\ast\sT M\times\sT^\ast\sT^\ast\R^\times$ is isomorphic, as a double vector bundle as well as a symplectic $\Rt$-principal bundle to
$$\sT\sT^\ast(M\times\R^\times)\simeq \sT\sT^\ast M \times \sT\sT^\ast\R^\times$$
by means of $\alpha_M^{-1}\times \beta_{\R^\times}^{-1}$. From the latter we have the projection onto $\sA\sT^\ast(M\times \R^\times)$ and further, by means of the anchor, onto $\sT\sJ^1(M\times \R^\ast)$,
\begin{small}
$$\xymatrix@C+2pt@R+2pt{
\sT\sT^\ast M \times \sT\sT^\ast\R^\times \simeq \sT\sT^\ast(M\times\R^\times)\ar[d] &
\sT^\ast(\sT M\times\sT^\ast\R^\times)\simeq \sT^\ast\sT M\times\sT^\ast\sT^\ast\R^\times\ar[l]_{\alpha_M^{-1}\times \beta_{\R^\times}^{-1}}\ar[d] \\
\sT\sT^\ast M\times\R^\ast\times \R\times\R^\ast\simeq\sA\sT^\ast(M\times \R^\times)\ar[d] & \sJ^1(\sT M\times\R^\ast\times\R)\simeq\sT^\ast(\sT M\times\R^\ast)\times \R\ar[l]_{\bar\alpha_0}\ar[dl]_{\bar\alpha} \\
\sT\sT^\ast M\times\R^\ast\times\R^\ast\simeq\sT\sJ^1(M\times \R^\ast) & }.$$
\end{small}
In coordinates the map $\bar\alpha$ reads
$$\bar\alpha(x^i, \dot x^i, b_{x^j}, b_{\dot x^j}; z,u, w)=(x^i,  b_{\dot x^j}, \dot x^j, b_{ x^j}-ub_{\dot x^j}, z,-w).$$
We can write it alternatively as
$$(x,p,\dot x, \dot p, z, \dot z)\circ\bar\za=(x^i,  b_{\dot x^j}, \dot x^j, b_{x^j}-ub_{\dot x^j}, z,-w).$$
Applying $\bar\alpha$ to the first jet of $F$, we get
$$(x,p,\dot x, \dot p, z, \dot z)\circ \bar\alpha(\sj^1 F)=\left(x^i,\frac{\partial F}{\partial \dot x}, \frac{\partial F}{\partial x}- \frac{\partial F}{\partial z}\frac{\partial F}{\partial \dot x}, z, -F\right).$$
This leads to second-order equations of the form
$$\frac{\xd}{\xd s}\left(\frac{\partial F}{\partial \dot x}\right)= \frac{\partial F}{\partial x}-\frac{\partial F}{\partial z}\frac{\partial F}{\partial \dot x},\qquad \dot z=-F.$$
Taking $\bar \ell=-F$ as an alternative Lagrangian, we get
$$\frac{d}{ds}\left(\frac{\partial \bar\ell}{\partial \dot x}\right)= \frac{\partial \bar\ell}{\partial x}+\frac{\partial \bar\ell}{\partial z}\frac{\partial {\bar\ell}}{\partial \dot x},\qquad \dot z=\bar\ell,$$
which are the so called \emph{Herglotz equations} (cf. \cite{deLeon:2021,Esen:2021,Herglotz:1930,Simoes:2020}. Note, however, that this version of contact Lagrangian mechanics works only for trivial contact structures.

\begin{example} The contact mechanics was proposed as a tool that is appropriate for dissipative mechanical systems. Let us consider a simple example of such a system. Let $L=M\times\R$ be a trivial line bundle, where $M$ is a Riemannian manifold with a metric $g$. With $G:\sT M\rightarrow \sT^\ast M$ we denote the isomorphism induced by $g$, and by $\Vert \cdot\Vert$ the corresponding norms in $\sT M$ and $\sT^*M$ (this should not lead to any confusion). Since $L$ is trivial, a Hamiltonian section is given by a function on $\sT^\ast M\times \R^\ast$. Let us consider (in the notation of Example \ref{ex:1}) the Hamiltonian
$$\mh(p,z)=\frac{\Vert p\Vert^2}{2m} +\lambda z= \mh_0(p)+\lambda z,$$
with $\mh_0$ denoting the Hamiltonian for a free motion,  and $\lambda$ being a positive constant. The first jet of $\mh$ reads then
$$\sj^1\mh=(\xd\mh_0(p), \mh(p,z),\lambda).$$
Using formula (\ref{e:14}), we get the dynamics
$$\mathcal{D}_{\mh}=\{(\beta_M^{-1}(\xd\mh_0(p))\oplus(-\lambda p), z, {\Vert p\Vert^2}-\mh):\;\; p\in\sT^\ast M, z\in\R \}.$$
Since the dynamics is generated by a single Hamiltonian, it is the image of the Hamiltonian vector field. In coordinates $(x^i, p_j, z)$ on $\sJ^1L^\ast\simeq \sT^\ast M\times \R^\ast$
it reads
$$X_\mh=\frac{1}{m}g^{ij}(x)p_j\frac{\partial}{\partial x^i}-\left(\frac{1}{2m}\frac{\partial g^{ij}}{\partial x^k}p_ip_j+\lambda p_k\right)\frac{\partial}{\partial p_k}+
\left(\frac{1}{2m}g^{ij}p_ip_j-\lambda z\right)\frac{\partial}{\partial z}.$$
The above vector field is projectable on $\sT^\ast M$ and the result is the Hamiltonian vector field for the free motion, corrected by the term $-\lambda p_k\frac{\partial}{\partial p_k}$, which is responsible for the additional force linearly depending on the momentum, e.g., a viscosity force.
\medskip

Applying the contact Legendre transformation, we get a family of functions on $\sT M\times \R$ parameterized by points of $\sT^\ast M \times\R$,
$$\sT^\ast M\times_M\sT M\times\R\times\R\ni(p,v,z,\tau)\longmapsto\langle p,v\rangle+z(\tau-\lambda)-\frac{{\Vert p\Vert^2}}{2m}\in\R.$$
This family can be simplified, since the vanishing of its differential in the direction of $\sT M$ gives the condition $v=\frac{1}{m}G^{-1}(p)$. The simplified family is parameterized by $z$ only,
$$\sT M\times\R\times\R\ni(v,\tau,z)\longmapsto\ell(v,\tau,z)=\ell_0(v)+z(\tau-\lambda)
=\frac{m{\Vert v\Vert^2}}{2} +z(\tau-\lambda),$$
where $\ell_0$ is the Lagrangian for the free motion. The above family generates the following Legendre submanifold in $\sJ^1L_{\sT P}$:
$$\sJ^1L_{\sT P}\simeq \sT^\ast\R\times\R\times\R\supset
\{(\xd\ell_0(v), \lambda, \ell_0(v), z)\,:\; v\in\sT M, z\in\R\}.$$
Applying $\alpha$, we get the dynamics expressed in a Lagrangian way, namely
\begin{align*}
&\sT(\sJ^1L_P)\simeq\sT\sT^\ast M\times \R\times\R\supset\mathcal{D}\\
&\qquad=
\big\{\alpha_M^{-1}\big(\xd\ell_0(v)\big)\oplus \big(-\lambda m G(v)), z, \ell_0(v)-\lambda z\big)\,:\; v\in \sT M, z\in\R\,  \big\}.
\end{align*}
In the adapted coordinates $(x^i, p_j,\dot x^k, \dot p_l, z,\dot z)$, we get
$$p_j=mg_{ij}\dot x^i,\qquad \dot p_k=\frac{m}{2}\frac{\partial g_{ij}}{\partial x^k}\dot x^i\dot x^j-\lambda mg_{ik}\dot x^i, \qquad \dot z=\frac{m}{2}g_{ij}\dot x^i\dot x^j-\lambda z.$$
From the above, we get immediately second-order equations for curves in $M\times \R$ in the form
$$\ddot x^i=-\Gamma^{i}_{jk}\dot x^j\dot x^k-\lambda \dot x^i, \qquad \dot z=\frac{m}{2}g_{ij}\dot x^i\dot x^j-\lambda z,$$
which are exactly the Herglotz equations for the Herglotz Lagrangian
$$\ell_H:\sT M\times\R\ni (v,z)\longmapsto \frac{m\Vert v\Vert^2}{2}-\lambda z\in\R.$$
\end{example}

\begin{example}
To construct a Hamiltonian for topologically nontrivial line bundles, let us go back to the first jet bundle $\sJ^1B$ of the M\"obius band $B\to S^1$ (cf. Example \ref{AB}).  It is equipped with two charts $\bar\cO$ and $\bar\cU$, and the corresponding coordinates which we denote here $(q,p,z)$ and $(q',p',z')$, respectively,  which transform as in (\ref{tr1}) and (\ref{tr2}):
\beas \bar{\psi}\circ\bar{\varphi}^{-1}(q,p,z)&=&(q,p,z)\quad\text{if}\quad q\in\Big]
\frac{\pi}{2}, \pi\Big[;\\
\bar{\psi}\circ\bar{\varphi}^{-1}(q,p,z)&=&(q+\pi,-p,-z)
\quad\text{if}\quad q\in\Big]0,\frac{\pi}{2}\Big[.
\eeas
Let us recall that the contact structure on $\sJ^1B$ is represented by the local contact forms
$\eta_\mathcal{O}=\xd z-p\,\xd q$ and $\eta_\mathcal{U}=\xd z'-p'\,\xd q'$. If $q\in\big]0,\frac{\pi}{2}\big[$, then we have $\eta_\mathcal{O}=-\eta_\mathcal{U}$, while for $q\in\big]\frac{\pi}{2},\pi\big[$ we have $\eta_\mathcal{O}=\eta_\mathcal{U}$, so this contact structure is not trivializable.

To use the Hamiltonian picture described in Section \ref{s4.1}, let us put $L=B^*$, so that Hamiltonians $\mh$ are represented by maps $\mh:\sJ^1B\to B$ covering the identity on $S^1$.
The naive `free' Hamiltonian $\mh$ represented in coordinates $(q,p,z)$ by the function $H_\cO(q,p,z)=p^2/2$
does not define such a map properly because, like in (\ref{F}), it should be
\be\label{cham}H_\cO(q+\pi,-p,-z)=-H_\cO(q,p,z)\ee
for $q\in]0,\frac{\zp}{2}[$. A simple example of a correctly defined (i.e., satisfying (\ref{cham})) Hamiltonian is $\mh$ represented by
\be\label{hex} H_\cO(q,p,z)=\cos(q)\frac{p^2}{2}+\zl z,\ee
but this Hamiltonian is not regular. Regular Hamiltonians do exist, as we have shown in Example \ref{regL}, however, they have a rather complicated form. In other words, topologically non-trivial `contact worlds' require much more advanced studies.

In any case, we can compute the contact vector field on $\sJ^1B$ corresponding to Hamiltonian (\ref{hex}) which locally reads (cf. (\ref{e:2}))
\be\label{dyn}X_\mh=\cos(q)p\pa_q+\Big(\sin(q)\frac{p^2}{2}-\zl p\Big)\pa_p+\Big(\cos(q)\frac{p^2}{2}-\zl z\Big)\pa_z.\ee
Of course, we also obtain correctly defined Hamiltonians by adding to $H_\cO$ from (\ref{hex}) any $\pi$-anti-symmetric `potential' $V(q)$, i.e., such that $V(q+\pi)=-V(q)$ (e.g., $V(q)=\sin(q)$).

Since the Hamiltonian given by (\ref{hex}) is not regular, there is no Lagrangian (i.e., a map $\sA (B^\ast)^\ti\rightarrow B$) generating dynamics (\ref{dyn}). There is however a Lagrangian generating object, namely the following map
$$\ell: \sA (B^\ast)^\ti\times_{S^1}\sJ^1B\ni(v,u)\longmapsto \Pi(v,u)-\mh(u)\in B,$$
which can be understood as a family of maps $\sA (B^\ast)^\ti\rightarrow B$ parameterized by elements of $\sJ^1B$. In the coordinates $(q,\dot q, t)$ in $\sA (B^\ast)^\ti$ (see Example \ref{regL}) and $(q,p,z)$ in $\sJ^1 B$, the generating family is given  by the function
\be\label{Lf}L_{\mathcal{O}}(q,\dot q,t,p,z)=p\dot q+tz-\cos(q)\frac{p^2}{2}-\lambda z.\ee
In expression (\ref{Lf}), $(p,z)$ are the parameters of the generating family.
\end{example}

\section{Concluding remarks}
We proposed in this paper an approach to the concept of a Tulczyjew triple for contact manifolds.
This approach is not an \emph{ad hoc} postulate but is canonically derived from the classical one.
This was possible because of changing the standard language of contact geometry into a symplectic one, in which contact structures on $M$ are understood as particular (homogeneous) symplectic structures on certain principal bundles $P\to M$ with the structure group $\Rt=\R\setminus\{ 0\}$ of multiplicative non-zero reals. This approach is valid for all possible (also non-trivializable) contact structures. Natural examples of such structures are the cotangent bundles $\sT^*\Lt$, where $\Lt$ is the $\Rt$-principal bundle consisting of nonzero vectors in a line bundle $L\to Q$, with the standard multiplication by nonzero reals and the canonical symplectic form. It turns out that the base of this principal bundle, $\sT^*\Lt/\Rt$, is canonically isomorphic to the bundle $\sJ^1L^*$ of first jets of sections of the dual line bundle $L^*\to Q$, which therefore carries a canonical contact structure. The contact manifolds $\sJ^1L^*$, which are canonically vector bundles over $Q$, are the only (up to contactomorphisms) linear contact structures, like the cotangent bundles $\sT^*Q$ are the only (up to symplectomorphisms) linear symplectic structures on vector bundles.

It is therefore not strange that contact manifolds $\sJ^1L^*$ play the r\^ole of phase spaces for contact Tulczyjew triples. We obtain (generally implicit) dynamics on $\sJ^1L^*$ from Hamiltonians or Lagrangians understood as sections of certain line bundles: over the phase space in the case of Hamiltonians, and over the so-called Atiyah algebroid $\sA\Lt$ for the $\Rt$-principal bundle $\Lt$ in the case of Lagrangians. The Legendre maps and Legendre transformations, as relating Hamiltonian and Lagrangian formalisms, have been naturally defined, together with a concept of hyperregularity of Hamiltonians or Lagrangians and some examples. As a by-product we have obtained canonical lifts of contact structures on $M$ to contact structures on the Atiyah algebroid $\sA P$ associated with an $\Rt$-principal bundle $P\to M$.

This is the first paper dealing seriously with the subject (in paper \cite{Esen:2021} some slightly different contact Tulczyjew triples are proposed, however, defined \emph{ad hoc} and only for trivial contact structures),
so there is still much work to be done in this direction, especially if examples with physics motivation and applications are concerned. Also, the question of the existence of hyperregular Lagrangians in the general case, as well as a proper understanding of `kinetic energy' (`free' Hamiltonians and Lagrangians) in topologically nontrivial cases, are interesting but probably complicated problems.






\small{\vskip1cm}

\small{\vskip1cm}

\noindent Katarzyna GRABOWSKA\\
Faculty of Physics \\
University of Warsaw\\
Pasteura 5,
02-093 Warszawa, Poland
\\Email: konieczn@fuw.edu.pl\\

\noindent Janusz GRABOWSKI\\ Institute of
Mathematics\\  Polish Academy of Sciences\\ \'Sniadeckich 8, 00-656 Warszawa, Poland
\\Email: jagrab@impan.pl \\

\begin{thebibliography}{V}

\bibitem{Benenti:2011} S.~Benenti,
\newblock{{Symplectic Relations on Symplectic Manifolds,}}
\newblock{in: Hamiltonian Structures and Generating Families. Universitext. Springer, New York, NY., 2011.}

\bibitem{Bravetti:2017} A.~Bravetti,
\newblock{\it{Contact Hamiltonian dynamics: the concept and its use,}}
\newblock{{Entropy} \textbf{19} (2017) 535, 22pp.}

\bibitem{Bravetti:2017a} A.~Bravetti, H.~Cruz, D.~Tapias,
\newblock{\it{Contact {H}amiltonian mechanics,}}
\newblock{{Ann. Phys.} \textbf{376} (2017), 17--39.}

\bibitem{Bruce:2017} A.~J.~Bruce, K.~Grabowska, J.~Grabowski,
\newblock{\it{Remarks on contact and Jacobi geometry,}}
\newblock{{SIGMA Symmetry Integrability Geom. Methods Appl.} \textbf{13} (2017), Paper No. 059, 22 pp.}


\bibitem{Ciaglia:2018} F.~M.~Ciaglia, H.~Cruz, G.~Marmo,
\newblock{\it{Contact manifolds and dissipation, classical and quantum,}}
\newblock{{Ann. Phys.} \textbf{398} (2018), 159--179.}

\bibitem{Cruz:2018} H.~Cruz,
\newblock{\it{Contact Hamiltonian mechanics. An extension of symplectic Hamiltonian mechanics,}}
\newblock{{J. Phys.: Conference Series} \textbf{1071} (2018) 012010.}

\bibitem{deLeon:2005} M.~de~Le\'on, J.~C.~Marrero, E.~Mart\'inez,
\newblock{\it{Lagrangian submanifolds and dynamics on Lie algebroids,}}
\newblock{{ J. Phys. A: Math. Gen} \textbf{38} (2005), R241--R308.}

\bibitem{deLeon:2017} M.~de~Le\'{o}n, C.~Sard\'{o}n,
\newblock{\it{Cosymplectic and contact structures for time-dependent and dissipative {H}amiltonian systems,}}
\newblock{{J. Phys. A} \textbf{50} (2017) 255205, 23pp.}

\bibitem{deLeon:2019} M.~de~Le\'on and M.~Lainz Valc\'azar,
\newblock{\it{Contact Hamiltonian systems,}}
\newblock{{J. Math. Phys.} \textbf{60} (2019) 102902, 18pp.}

\bibitem{deLeon:2020} M.~de~Le\'on, J.~Gaset, M.~Lainz Valc\'azar, X.~Rivas, N.~Rom\'an-Roy,
\newblock{\it{Unified Lagrangian-Hamiltonian formalism for contact systems,}}
\newblock{{Fortschr. Phys.} \textbf{68} (2020) 2000045, 12 pp.}

\bibitem{deLeon:2021} M.~de~Le\'on, M.~La\'{\i}nz, M.~C.~Mu\~noz-Lecanda, N.~Rom\'an-Roy,
\newblock{\it{Constrained Lagrangian dissipative contact dynamics,}}
\newblock{{J. Math. Phys.} \textbf{62} (2021) 122902.}

\bibitem{deLeon:2021a} M.~de~Le\'on, V.~M.~Jim\'enez, M.~Lainz,
\newblock{\it{Contact Hamiltonian and Lagrangian systems with nonholonomic constraints,}}
\newblock{{J. Geom. Mech.} \textbf{13} (2021), 25--53.}


\bibitem{Esen:2021} O.~Esen, M.~Lainz Valcazar, M.~de Leon, J.~C.~Marrero,
\newblock{\it{Contact Dynamics: Legendrian and Lagrangian Submanifolds,}}
\newblock{{Mathematics} \textbf{9}, (2021) 2704.}

\bibitem{Gaset:2020} J.~Gaset, X.~Gr\`acia, M.~C.~Mu\~noz-Lecanda, X.~Rivas, N.~Rom\'an-Roy,
\newblock{\it{New contributions to the Hamiltonian and Lagrangian contact formalisms for dissipative
mechanical systems and their symmetries,}}
\newblock{\emph{Int. J. Geom. Methods Mod. Phys.} \textbf{17} (2020) 2050090, 27 pp.}

\bibitem{Grabowska:2010} K.~Grabowska,
\newblock{\it{Lagrangian and Hamiltonian formalism in Field theory: a simple model,}}
\newblock{{J. Geom. Mech.} \textbf{2} (2010), 375--395.}

\bibitem{Grabowska:2012} K.~Grabowska,
\newblock{\it{A Tulczyjew triple for classical fields,}}
\newblock{{J. Phys. A: Math. Theor.} \textbf{45} (2012) 145207.}

\bibitem{Grabowska:2008} K.~Grabowska, J.~Grabowski,
\newblock{\it{Variational calculus with constraints on general algebroids,}}
\newblock{{J. Phys. A} \textbf{41} (2008) 175204, 25 pp.}

\bibitem{Grabowska:2011} K.~Grabowska, J.~Grabowski,
\newblock{\it{Dirac algebroids in Lagrangian and Hamiltonian mechanics,}}
\newblock{{J. Geom. Phys.} \textbf{61} (2011), 2233--2253.}

\bibitem{Grabowska:2013} K.~Grabowska, J.~Grabowski,
\newblock{\it{Tulczyjew triples: from statics to field theory},}
\newblock{{J. Geom. Mech.} \textbf{5} (2013), 445--472.}

\bibitem{Grabowska:2022} K.~Grabowska, J.~Grabowski,
\newblock{\it{A novel approach to contact Hamiltonians and contact Hamilton-Jacobi Theory.}}
\newblock{{J. Phys. A} \textbf{55} (2022) 435204 (34pp).}

\bibitem{Grabowska:2023} K.~Grabowska, J.~Grabowski,
\newblock{\it{Reductions: precontact versus presymplectic,}}
\newblock{{Ann. Mat. Pura Appl.} \textbf{202} (2023), 2803--2839.}


\bibitem{Grabowska:2004} K.~Grabowska, J.~Grabowski, P. Urbański,
\newblock{\it{AV-differential geometry: Poisson and Jacobi structures,}}
\newblock{{J. Geom. Phys.} \textbf{52} (2004), 389--446.}

\bibitem{Grabowska:2006a} K.~Grabowska, P. Urbański,
\newblock{\it{AV-differential geometry and Newtonian mechanics,}}
\newblock{{Rep. Math. Phys.} \textbf{58} (2006), 21--40.}

\bibitem{Grabowska:2006} K.~Grabowska, P.~Urbański, J.~Grabowski,
\newblock{\it{Geometrical mechanics on algebroids,}}
\newblock{{Int. J. Geom. Methods Mod. Phys.} \textbf{3} (2006), 559--575.}

\bibitem{Grabowska:2015} K.~Grabowska, L. Vitagliano,
\newblock{\it{Tulczyjew triples in higher derivative field theory,}}
\newblock{{J. Geom. Mech.} \textbf{7} (2015), 1--33.}

\bibitem{Grabowska:2016}
K.~Grabowska, M.~Zaj\c{a}c,
\newblock{\it{The {T}ulczyjew triple in mechanics on a {L}ie group,}}
\newblock{{J. Geom. Mech.} \textbf{8} (2016), 413--435.}

\bibitem{Grabowski:2013} J.~Grabowski,
\newblock{\it{Graded contact manifolds and contact Courant algebroids}}
\newblock{{J. Geom. Phys. } \textbf{68} (2013), 27--58.}

\bibitem{Grabowski:2009} J.~Grabowski, M.~Rotkiewicz,
\newblock{\it{Higher vector bundles and multi-graded symplectic manifolds,}}
\newblock{{J. Geom. Phys.} \textbf{59} (2009), 1285--1305.}

\bibitem{Grabowski:2012} J.~Grabowski, M.~Rotkiewicz,
\newblock{\it{Graded bundles and homogeneity structures,}}
\newblock{{J. Geom. Phys.} \textbf{62} (2012), 21--36.}

\bibitem{Grabowski:1995} J.~Grabowski, P.~Urba\'nski,
\newblock{\it{Tangent lifts of Poisson and related structures,}}
\newblock{{J. Phys. A}, \textbf{28} (1995), 6743--6777.}

\bibitem{Herglotz:1930} G.~Herglotz,
\newblock{Ber\"uhrungstransformationen,}
\newblock{Lectures at the University of G\"ottingen; University of G\"ottingen:
G\"ottingen, Germany, 1930.}

\bibitem{Konieczna:1999} K.~Konieczna,  P.~Urba\'nski,
\newblock{\it{Double vector bundles and duality,}}
\newblock{{Arch. Math. (Brno)} \textbf{35} (1999), 59--95.}

\bibitem{Le:2018} H\^ong V\^an L\^e, Yong-Geun Oh, A.~G.~Tortorella, L.~Vitagliano,
\newblock{\it{Deformations of coisotropic submanifolds in Jacobi manifolds,}}
\newblock{{J. Symplectic Geom.} \textbf{16} (2018), 1051--1116.}

\bibitem{Libermann:1996} P.~Libermann,
\newblock{\it{Lie algebroids and mechanics,}}
\newblock{{Arch. Math. (Brno)} \textbf{32} (1996), 147--162.}

\bibitem{Martinez:2001} E. Mart{\'\i}nez,
\newblock{\it{Lagrangian mechanics on Lie algebroids,}}
\newblock{{Acta Appl. Math.} \textbf{67} (2001), 295--320.}

\bibitem{Pradines:1974} J.~Pradines,
\newblock{\it{Repr\'{e}sentation des jets non holonomes par des morphismes vectoriels doubles soud\'{e}s,}}
\newblock{{C.R. Acad. Sci. Paris, s\'{e}rie A} \textbf{278} (1974), 1523--1526.}

\bibitem{Simoes:2020}
A.~A.~Simoes, M.~de~Le{\'o}n, M.~Lainz~Valc{\'a}zar, D.~Mart{\'\i}n~de Diego,
\newblock{\it{Contact geometry for simple thermodynamical systems with friction,}}
\newblock{{Proc. R. Soc. A.} \textbf{476} (2020) 20200244.}


\bibitem{Tulczyjew:1974} W.~M.~Tulczyjew,
\newblock{\it{Hamiltonian Systems, Lagrangian systems and the Legendre transformation,}}
\newblock{{Symp. Math.} \textbf{14}, Roma (1974), 247--258.}

\bibitem{Tulczyjew:1976a} W.~M.~Tulczyjew,
\newblock{\it{Les sous-varietes Lagrangiennes et la Dynamique Hamiltonienne,}}
\newblock{{C. R. Acad. Sci. Paris} \textbf{283}, Roma (1976), 15--23.}

\bibitem{Tulczyjew:1976b} W.~M.~Tulczyjew,
\newblock{\it{Les sous-varietes Lagrangiennes et la Dynamique Lagrangienne,}}
\newblock{{C. R. Acad. Sci. Paris} \textbf{283}, Roma (1976), 675--683.}

\bibitem{Tulczyjew:1977} W.~M.~Tulczyjew,
\newblock{\it{The Legendre transformation},}
\newblock{{Ann. Inst. H. Poincar\'e, Sect. A}, \textbf{27} (1977), 101--114.}

\bibitem{Tulczyjew:1989} W.~M.~Tulczyjew,
\newblock{\it{Geometric Formulation of Physical Theories,}}
\newblock{{Bibliopolis} Naples (1989).}

\bibitem{Tulczyjew:1999} W.~M. Tulczyjew, P.~Urba\'{n}ski,
\newblock{\it{A slow and careful {L}egendre transformation for singular {L}agrangians,}}
\newblock{{Acta Polonica} \textbf{30} (1999), The Infeld
Centennial Meeting (Warsaw, 1998), pp.~2909--2978.}

\bibitem{Urbanski:2003} P.~Urba\'{n}ski,
\newblock{\it{An affine framework for analytical mechanics, Classical and quantum integrability}}
\newblock{(Warsaw, 2001),{Banach Center Publ.} \textbf{59}, Polish Acad. Sci. Inst. Math., Warsaw 2003, 257--279.}

\bibitem{Weinstein:1996} A. Weinstein,
\newblock{\it{Lagrangian mechanics and groupoids,}}
\newblock{{Fields Inst. Comm.} \textbf{7} (1996), 207–231.}

\bibitem{Zapata:2020} C.~Zapata-Carratal\'a,
\newblock{\it{Jacobi geometry and Hamiltonian mechanics: the unit-free approach,}}
\newblock{{Int. J. Geom. Methods Mod. Phys.} \textbf{17} (2020) 2030005.}

\end{thebibliography}
\end{document}